\documentclass[a4paper,UKenglish,cleveref, autoref,numberwithinsect,nolineno]{socg-lipics-v2021}

\usepackage{graphicx}
\usepackage{color}
\usepackage{times}
\usepackage{amsmath}
\usepackage{amssymb}
\usepackage{xspace}
\usepackage{epsfig}
\usepackage[dvipsnames]{xcolor}
\usepackage[colorinlistoftodos,prependcaption,textsize=tiny]{todonotes}

\newcommand {\mm}[1] {\ifmmode{#1}\else{\mbox{\(#1\)}}\fi}

\newcommand {\floor}[1] {{\left\lfloor #1 \right\rfloor}}

\newcommand{\abs}[1]       {\mm{\left|{#1}\right|}}

\newsavebox{\smallProofsym}                            
\savebox{\smallProofsym}                               %
{
\begin{picture}(6,6)                                   %
\put(0,0){\framebox(6,6){}}                            %
\put(0,2){\framebox(4,4){}}                            %
\end{picture}                                          %
}                                                      %
\makeatletter
\long\def\@makecaption#1#2{%
  \vskip\abovecaptionskip
  \sbox\@tempboxa{\small #1: #2}%
  \ifdim \wd\@tempboxa >\hsize
    \small #1: #2\par
  \else
    \global \@minipagefalse
    \hb@xt@\hsize{\hfil\box\@tempboxa\hfil}%
  \fi
  \vskip\belowcaptionskip}
\makeatother

\newcommand{\Rspace}        {\mm{{\mathbb R}}}
\newcommand{\barRspace}     {\mm{\bar{\mathbb R}}}
\newcommand{\Zspace}        {\mm{{\mathbb Z}}}
\newcommand{\filter}        {\mm{F}}
\newcommand{\filteraux}     {\mm{G}}
\newcommand{\MTree}         {\mm{{\mathbb M}}}
\newcommand{\perMTree}[2]   {\mm{{\mathbb M}{({#1},{#2})}}}
\newcommand{\MTreeIndex}    {\mm{\MTree_\textrm{index}}}
\newcommand{\NTree}         {\mm{{\mathbb N}}}
\newcommand{\OTree}         {\mm{{\mathbb O}}}
\newcommand{\BCode}         {\mm{{\mathbb B}}}
\newcommand{\perBCode}[2]   {\mm{{\mathbb B}_{0}{({#1},{#2})}}}
\newcommand{\slSpace}[2]    {\mm{[{#2}]}}
\newcommand{\Height}        {\mm{{\tt{h}}}}
\newcommand{\Frequency}     {\mm{{\tt{f}}}}
\newcommand{\Shadow}[2]     {\mm{{\tt{m}}_{#1}{({#2})}}}
\newcommand{\Monomial}      {\mm{{\mathcal M}}}
\newcommand{\multiplicity}  {\mm{\mu}}

\newcommand{\GCD}[2]        {\mm{{\rm gcd}{({#1},{#2})}}}
\newcommand{\bigOh}[1]      {\mm{\mathcal{O}({#1})}}
\newcommand{\Wasser}[3]     {\mm{{W_{#1}{({#2},{#3})}}}}
\newcommand{\Wasserpm}[3]   {\mm{{W_{#1}^{\pm}{({#2},{#3})}}}}
\newcommand{\volume}[2]     {\mm{{{\rm vol}_{#1}\,}{#2}}}
\newcommand{\Span}[1]       {\mm{{\rm span}{({#1})}}}
\newcommand{\support}[1]    {\mm{{\rm supp}{({#1})}}}
\newcommand{\Tplan}         {\mm{T}}
\newcommand{\Xplan}         {\mm{X}}
\newcommand{\Yplan}         {\mm{Y}}

\newcommand{\card}[1]       {\mm{{\#}{#1}}}

\newcommand{\dime}[1]       {\mm{\rm dim\,}{#1}}

\newcommand{\degree}[1]     {\mm{\rm deg\,}{#1}}
\newcommand{\norm}[1]       {\mm{\|{#1}\|}}
\newcommand{\normtwo}[1]    {\mm{\|{#1}\|_2}}
\newcommand{\normop}[2]     {\mm{\|{#1}\|_{\rm op}^{#2}}}

\newcommand{\Maxdist}[2]    {\mm{\|{#1}-{#2}\|_\infty}}
\newcommand{\Onedist}[2]    {\mm{\|{#1}-{#2}\|_1}}
\newcommand{\Idist}[2]      {\mm{{I}{({#1},{#2})}}}
\newcommand{\Jdist}[2]      {\mm{{J}{({#1},{#2})}}}

\newcommand{\ee}            {\mm{\varepsilon}}
\newcommand{\aaa}           {\mm{\alpha}}
\newcommand{\bbb}           {\mm{\beta}}
\newcommand{\ccc}           {\mm{\gamma}}
\newcommand{\AAA}           {\mm{\mathrm A}}
\newcommand{\BBB}           {\mm{\mathrm B}}
\newcommand{\CCC}           {\mm{\Gamma}}

\newcommand{\Root}[1]       {\mm{{\rm Root}{({#1})}}}
\newcommand{\Size}[1]       {\mm{{\rm Size}{({#1})}}}
\newcommand{\Shift}[1]      {\mm{{\rm Shift}{({#1})}}}
\newcommand{\Drift}[1]      {\mm{{\rm Drift}{({#1})}}}
\newcommand{\Next}[1]       {\mm{{\rm Next}{({#1})}}}
\newcommand{\Basis}[1]      {\mm{{\rm Basis}{({#1})}}}
\newcommand{\Oldest}[1]     {\mm{{\rm Old}{({#1})}}}

\newcommand{\Skip}[1]       {}

\title{Merge Trees of Periodic Filtrations}

\authorrunning{Herbert Edelsbrunner, Teresa Heiss}

\author{Herbert Edelsbrunner}{ISTA (Institute of Science and Technology Austria), Kloster\-neu\-burg, Austria}{Herbert.Edelsbrunner@ist.ac.at}{https://orcid.org/0000-0002-9823-6833}{}
\author{Teresa Heiss}{ISTA (Institute of Science and Technology Austria), Kloster\-neu\-burg, Austria}{Teresa.Heiss@ist.ac.at}{https://orcid.org/0000-0002-1780-2689}{}

\Copyright{Herbert Edelsbrunner, Teresa Heiss}



\funding{Both authors are partially supported by the European Research Council (ERC) Horizon 2020 project `Alpha Shape Theory Extended', grant no.\ 788183.
The first author is also partially supported by the DFG Collaborative Research Center TRR 109, `Discretization in Geometry and Dynamics', Austrian Science Fund (FWF), grant no.\ I 02979-N35.}

\nolinenumbers

\begin{document}
\maketitle

\begin{abstract}
  Motivated by applications to crystalline materials, we generalize the merge tree and the related barcode of a filtered complex to the periodic setting in Euclidean space.
  They are invariant under isometries, 
  changing bases, and indeed changing lattices.
  In addition, we prove stability under perturbations and provide an algorithm that under mild geometric conditions typically satisfied by crystalline materials takes $\bigOh{(n+m) \log n}$ time, in which $n$ and $m$ are the numbers of vertices and edges in the quotient complex, respectively.
\end{abstract}

\section{Introduction}
\label{sec:1}

The main objective of this paper is the extension of persistent homology to periodic data in Euclidean space.
Such data frequently arises in material science, in particular the study of crystalline materials, and we refer to Fedorov~\cite{Fed91} for an early attempt to classify such structures.
Modeling the atoms of a crystalline material as points leads to the mathematical concept of a \emph{periodic set}, which we define as the Minkowski sum of a finite set and a lattice; see e.g.\ \cite{Cas97,Zhi15}.
We study the case in which the periodic data is given within a finite parallelepiped; that is: the \emph{unit cell} of the lattice.

\smallskip
Persistent homology describes the connected components and holes of shapes that are filtered, meaning their parts appear at different times.
If this shape is a complex---perhaps the Delaunay triangulation of a point set---there are various ways to derive a filter, including the radius function on the Delaunay triangulation; see e.g.\ \cite[Chapter III.4]{EdHa10}.
Taking the persistence diagram of this function, we get a summary of the topological connectivity on all scales.
In the last two decades, this pipeline of going from the data to a filtered complex and further to the persistence diagram has become a popular tool in the topological approach to data analysis; see e.g.\ \cite{Car09,Wei11}.
While periodic data is abundant and would benefit from such an analysis, persistent homology has not yet been properly extended to periodic sets.
This paper takes a step in this direction, by generalizing the $0$-dimensional aspects of persistent homology to the periodic setting.
We list the specific contributions of this paper:
\smallskip \begin{itemize}
  \item the introduction of the \emph{shadow monomial} as the crucial ingredient to generalize the merge tree and the related $0$-th barcode or persistence diagram to the periodic setting;
  \item an algorithm that under usually satisfied conditions constructs the periodic merge tree and $0$-th barcode in $\bigOh{(n+m) \log n}$ time, with $n$ and $m$ the number of vertices and edges in a unit cell;
  \item the proof of invariance under changing the underlying lattice, both for periodic $0$-th barcodes and for the classes of periodic merge trees defined by the splintering equivalence relation;
  \item the extension of the Wasserstein distance between barcodes and the interleaving distance between merge trees to the periodic setting, and proofs of stability for both extensions.
\end{itemize} \smallskip
Our method can be applied to data in any dimension that can be viewed as a periodic graph with a periodic filter on it.
Similar to classical persistent homology, this includes but is not limited to
\smallskip \begin{itemize}
  \item \emph{locally finite periodic sets}, which generate filtered complexes (e.g.\ possibly weighted Delaunay) by growing balls centered at the points and monitoring when and where they overlap;
  \item \emph{periodic digital images}, in which the grayscale or color values in the pixels provide a filter, and the graph could either be the edge-skeleton of the cubical complex, or its dual;
  \item \emph{periodic functions} (e.g.\ piecewise linear or smooth), for which the graph structure and filter are induced by the critical points and their values at which the sublevel set alters its topology.
\end{itemize} \smallskip
Even aperiodic materials or the aperiodic universe are often simulated with enforced periodic boundary conditions to prevent boundary effects from contaminating the results gathered within the bulk. 

\smallskip
Prior approaches to persistent homology for periodic data relied primarily on one of two heuristics: work with a sufficiently large finite subset of the data, e.g.\ \cite{HNHEMN16,KHM20}, or compactify using the torus topology, which is a classic mathematical concept and supported by modern software, including the implementations of Delaunay triangulations in CGAL \cite{CGAL}.
The first heuristic suffers from unwanted boundary effects and unnecessarily high computation costs, while the latter loses information about the patterns in the periodic setting.
Our approach is related to both and reaps the benefits of the compactness of the torus while carefully recording the relation to the infinite periodic setting; see Figure~\ref{fig:snakes_overview}.
Indeed, each connected component on the torus has preimages (shadows) in the Euclidean space---perhaps infinitely many---and similar to Onus and Robins~\cite{OnRo22}, we count these preimages but in a fine-grained manner that keeps track of growth-rates as well as densities (exponents and coefficients of the shadow monomials).
The growth-rate distinguishes between connected clumps, strings, sheets, and blocks, while the density quantifies how densely these are distributed.
Both pieces of information are needed for detailed statements about how materials are composed.
The periodicity lattice---which is used in the definition of shadow monomials---has been introduced before under different names, and used to compute $0$-homology \cite{CoMe91,DiDu89,OnRo22}.
However, using it to quantify the growth-rate and density of the components is to the best knowledge of the authors new.

\begin{figure}[htp]
  \centering
  \includegraphics[width=0.39\linewidth]{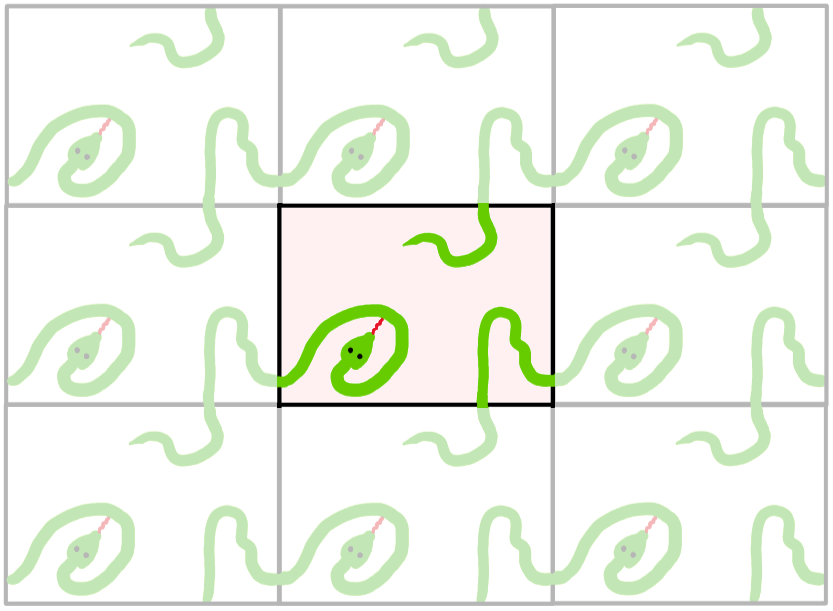}\hspace{0.3in}\includegraphics[width=0.39\linewidth]{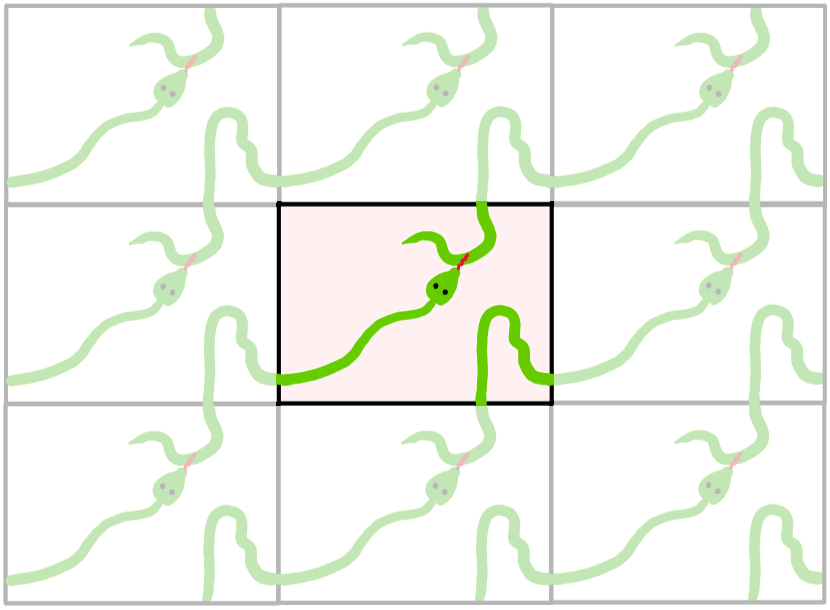}
  \caption{\footnotesize In a unit cell with periodic boundary conditions (the torus), we see a single snake that bites itself, both in the \emph{left} and the \emph{right panel}.
  There is however a significant difference in the periodically tiled plane, since the snakes on the \emph{right} connect in infinite diagonal lines, while the snakes on the \emph{left} remain isolated, a distinction we will quantify with the novel concept of a \emph{shadow monomial}.
  The material properties of the two examples would indeed be rather different, with higher resistance to tearing on the \emph{right}.}
  \label{fig:snakes_overview}
\end{figure}

\smallskip
Shifting the focus from counting components to recording when components merge, we mention the work of Ingrid Hotz and her group \cite{TASMH23}.
In collaborations with domain scientists, they study the electron density of $3$-dimensional layered materials, and found that the timing of the merge events is useful in the quick comparison of materials.
They compromise the (infinite) periodic setting by considering merge events within an array of $2 \times 2 \times 2$ unit cells with periodic boundary conditions.
Our method removes the restriction to a finite portion of the data and provides extra information, such as how many and which components join at a merge event.

\smallskip \noindent \textbf{Outline.}
Section~\ref{sec:2} extends the concept of a merge tree to the periodic setting.
Section~\ref{sec:3} explains the algorithm for constructing periodic merge trees.
Section~\ref{sec:4} proves the invariance and stability of equivalence classes of periodic merge trees.
Section~\ref{sec:5} introduces the periodic $0$-th barcode and proves invariance and stability.
Section~\ref{sec:6} illustrates the concepts using a $3$-dimensional periodic graph as an example.
Section~\ref{sec:7} concludes this paper.

\section{The Periodic Merge Tree}
\label{sec:2}

This paper is written for arbitrary but fixed dimension, $d$, with $d=3$ being the most important case for applications, including to crystalline materials.
We limit ourselves to the discrete setting of complexes or graphs, and stress that the ideas also apply to smooth and piecewise smooth functions.

\subsection{Merge Tree}
\label{sec:2.1}

We begin with standard definitions.
A \emph{cell complex}, $K$, in $\Rspace^d$ consists of points and cells that have these points as vertices.
We require that each cell, $\tau$, be homeomorphic to the $p$-dimensional closed unit ball, for some $0 \leq p \leq d$, and the boundary of $\tau$ is a union of lower-dimensional cells in the complex, called \emph{faces} of $\tau$.
In addition, we require that intersections of cells in the complex are either empty or a union of cells in their shared boundaries.
Depending on the context, we treat $K$ as a combinatorial object (a collection of cells with face relations), or as a geometric object (the set of points in the union of these cells).
We allow for infinitely many points and cells but require that $K$ be locally finite.
The points are referred to as the \emph{vertices}, and the $1$-dimensional cells as the \emph{edges} of $K$.
The complex is \emph{connected} if the graph that consists of its vertices and edges is connected.
Since this paper is primarily concerned with $0$-dimensional homology---whose classes correspond to the connected components---we will mostly ignore cells of dimension $2$ and higher.
After removing such cells, we are left with a graph consisting of vertices and edges, allowing for edges that start and end at the same vertex, and multiple edges connecting the same two vertices.
However, in anticipation of an extension of our methods to higher homology dimensions, and because the overhead is modest, we nonetheless write this paper using the more general terminology of complexes.

\smallskip
A \emph{filter} of $K$ is a function $\filter \colon K \to \Rspace$ that satisfies $\filter(\sigma) \leq \filter(\tau)$ whenever $\sigma$ is a face of $\tau$.
Given $t \in \Rspace$, the corresponding \emph{sublevel set} consists of all cells with value at most $t$, denoted $K_{t} = \filter^{-1} (-\infty, t]$.
By definition of filter, $K_{t} \subseteq K$ is itself a complex and thus a \emph{subcomplex} of $K$.
More generally, $K_s$ is a subcomplex of $K_t$, whenever $s \leq t$, which implies that every connected component of $K_{s}$ includes into a connected component of $K_{t}$.
The ordered sequence of different sublevel sets is the \emph{filtration} defined by $\filter$.
Think of the sublevel set as an object that evolves as the threshold grows continuously.
With this picture in mind, the merge tree of $\filter$ keeps track of the connected components and how they merge.
Intuitively, this tree has a point for each connected component of each sublevel set.
We draw the tree from left to right, nevertheless calling the threshold for the sublevel set the \emph{height} of the point.
As long as a connected component stays the same or just grows, we draw the point along a left-to-right trajectory, but when components merge, then the corresponding trajectories meet, and afterwards continue as a single trajectory again in left-to-right direction.
Since trajectories meet but never bifurcate, and---assuming $K$ is connected---eventually combine to a single trajectory, we indeed have a tree, as suggested by the name.
The number of points in a vertical section of this tree is the number of connected components of the corresponding sublevel set, and the connectivity in the left-to-right direction reflects the inclusion between connected components in different sublevel sets; see Figure~\ref{fig:TwoTreesOne}, which shows two merge trees of the same periodic set.
Each tree is drawn as a collection of horizontal intervals, called \emph{beams}, which we connect with short vertical line segments to represent the merging of two components.
\begin{figure}[hbt]
  \centering \vspace{0.0in}
  \resizebox{!}{1.75in}{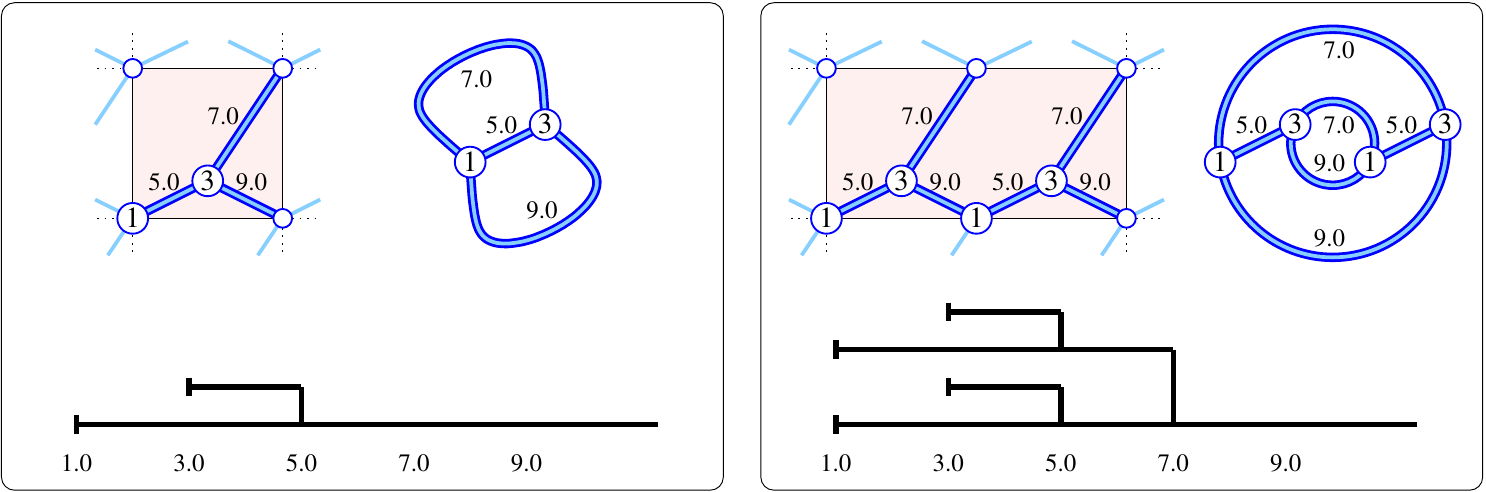}
  \caption{\footnotesize \emph{Left panel:} a periodic graph with two vertices and three edges inside a unit cell in the shape of the unit square in the \emph{upper left portion}, and the corresponding graph in the $2$-dimensional torus in the \emph{upper right portion of the panel}.
  The filter maps the vertices to their (real) labels and the edges to the values shown, which defines the merge tree at the \emph{bottom in the panel}.
  \emph{Right panel:} the same periodic graph as in the \emph{left panel}, but now represented by a sublattice with a rectangular unit cell of twice the area.
  Correspondingly, the graph in the $2$-dimensional torus has twice the number of vertices and edges, and the merge tree is richer than in the \emph{left panel}.}
  \label{fig:TwoTreesOne}
\end{figure}

\smallskip
Note the asymmetry, in which we always draw the vertical segment from the shorter to the longer beam (comparing their lengths to the left of the bifurcation point).
As a result, the shorter beam ends, and the longer beam continues without interruption to the right.
While not being part of the definition of merge tree, this asymmetry will come handy when we transition from merge trees to barcodes in Section~\ref{sec:5}.
We call this the \emph{elder rule} according to which the younger component gets absorbed into the older component; see \cite[page~150]{EdHa10}.
Ties are broken arbitrarily, and we will see in Section~\ref{sec:5} that different ways to break a tie does not affect the barcode.

\smallskip
To give a more formal definition of the merge tree, we introduce the \emph{sublevel space} of $\filter$, denoted $\slSpace{K}{\filter}$, which consists of all points $(x, s) \in K \times \Rspace$ with $x \in K_{s}$,
in which $x$ may be a vertex or a point on an edge or higher-dimensional cell in $K$.
Call two of its points \emph{equivalent}, written $(x,s) \sim (y,t)$, if $s=t$ and $x, y$ belong to the same connected component of $K_{s} = K_{t}$.
The \emph{quotient}, denoted $\slSpace{K}{\filter}_\sim$, effectively combines all equivalent points into one and inherits the \emph{quotient topology} from the Euclidean topology of $\Rspace^{d+1}$ restricted to $\slSpace{K}{\filter} \subseteq \Rspace^{d+1}$.
\begin{definition}[Merge Tree]
  \label{dfn:merge_tree}
  The \emph{merge tree} of a filter $\filter \colon K \to \Rspace$ is the quotient of the sublevel space of $\filter$, denoted $\MTree{} = \MTree{(\filter)} = \slSpace{K}{\filter}_\sim$, together with the corresponding quotient topology.
  It is equipped with the \emph{height function}, $\Height \colon \MTree \to \Rspace$, which maps each point to the value of the sublevel set in which the point is a connected component.
\end{definition}
Whenever we talk about a merge tree, we will tacitly assume that it comes 
with a height function.
We will furthermore simplify language by saying that points $\aaa$ and $\bbb$ of the merge tree not only represent but in fact \emph{are} connected components of their respective sublevel sets, namely of $K_{\Height(\aaa)}$ and $K_{\Height(\bbb)}$, respectively.
We say $\bbb$ \emph{covers} $\aaa$ if $\Height(\aaa) \leq \Height(\bbb)$ and $\aaa \subseteq \bbb$.

\subsection{Periodic Setting}
\label{sec:2.2}

We are specifically interested in filters that are periodically repeated on a complex that consists of periodically arranged copies of a finite complex.
The natural language for such a setting is that of lattices in Euclidean space, discussed e.g.\ in \cite{Cas97,Zhi15}.
Let $u_1, u_2, \ldots, u_p$ be linearly independent vectors in $\Rspace^d$.
The set of all integer combinations is the \emph{lattice} spanned by these vectors:
\begin{align}
  \Lambda &= \Lambda (u_1, u_2, \ldots, u_p)
           = \left\{ \lambda_1 u_1 + \lambda_2 u_2 + \ldots + \lambda_p u_p \mid \lambda_i \in \Zspace \mbox{\rm ~for~} 1 \leq i \leq p \right\} ,
\end{align}
and we call $p = \dime{\Lambda}$ the \emph{dimension} of the lattice.
The \emph{span} of the vectors, denoted $\Span{\Lambda}$, is the set of all real combinations of the vectors, which is a $p$-dimensional linear subspace of $\Rspace^d$.
Similarly, the \emph{unit cell} of the vectors is the set of real combinations with coefficients in $[0,1)$.
The same lattice can be spanned by different sets of vectors, but the dimension, the span, and the ($p$-dimensional) volume of the unit cell are always the same.
We write $\volume{p}{\Lambda}$ for the latter and note that for $p=d$, it is the determinant of the matrix whose columns are the vectors.
Assuming $\dime{\Lambda} = d$, we write $U$ for the $d \times d$ matrix whose columns are the basis vectors.
Since $U^{-1} U$ is the identity matrix, $U^{-1}$ maps $\Lambda$ to the standard integer lattice, $\Zspace^d$, and $U^{-1}$ applied to a lattice vector gives the integer coordinates of that vector with respect to the basis given by $U$.

\smallskip
Given a $d$-dimensional lattice $\Lambda$, a locally finite complex, $K \subseteq \Rspace^d$, is \emph{$\Lambda$-periodic} if $\sigma \in K$ and $u \in \Lambda$ implies $\sigma + u \in K$. 
Similarly, a filter on a $\Lambda$-periodic complex $K$ is a \emph{$\Lambda$-periodic} filter, $\filter \colon K \to \Rspace$, if it satisfies $\filter(\sigma) = \filter(\sigma + u)$ for every $\sigma \in K$ and $u \in \Lambda$. 
Periodicity means that information is repeated infinitely often. 
To avoid infinitely large merge trees, it can therefore be useful to reduce redundancy by building quotients.
The \emph{quotient}, $\Rspace^d / \Lambda$, is obtained by identifying points $x, y \in \Rspace^d$ whenever $y-x \in \Lambda$.
This quotient with topology inherited from the $d$-dimensional Euclidean space is usually referred to as the \emph{$d$-dimensional torus}.
Similarly, we can identify cells of $K$ if they are translates of each other by a vector of $\Lambda$. The resulting \emph{quotient complex}, $K/\Lambda$, is a complex on the $d$-torus consisting of finitely many cells because $K$ is locally finite and the $d$-torus is compact.\footnote{Note, however, that even a periodic complex that is simplicial can have a non-simplicial quotient, which includes the possibilities of more than two edges connecting the same two vertices and an edge connecting a single vertex back to itself; see Figure~\ref{fig:TwoTreesOne} for an example.}
We write $\filter/\Lambda \colon K/\Lambda \to \Rspace$ for the \emph{quotient filter}, which is a convenient representaton of $\filter \colon K \to \Rspace$.
Unfortunately, this representation is not unique since, for example, $K$ and $\filter$ are also periodic with respect to $2 \Lambda$, or really any sublattice of $\Lambda$; compare the two panels in Figure~\ref{fig:TwoTreesOne} in which the respective unit cells (a square on the left and a rectangle of twice the area on the right) define the same periodic complex
but yield different merge trees.
This motivates the introduction of a quantified version of the merge tree; see Figure~\ref{fig:TwoTreesTwo}.
We first explain how we quantify and then how we apply this idea to modify the merge tree.

\subsection{Shadow Monomial}
\label{sec:2.3}

Let $\Lambda \subseteq \Rspace^d$ be a $d$-dimensional lattice, and $\filter \colon K \to \Rspace$ a $\Lambda$-periodic filter on a $\Lambda$-periodic complex.
Write $\phi \colon \Rspace^d \to \Rspace^d / \Lambda$ for the projection that maps every $x$ to $x + \Lambda$.
Taking the inverse, a sublevel set of the quotient filter lifts to the corresponding sublevel set of the filter: $K_t = \phi^{-1} ((K / \Lambda)_t)$.
Given a component $\CCC \subseteq (K / \Lambda)_t$, we call a component $\ccc$ of $\phi^{-1}(\CCC)$ a \emph{shadow} of $\CCC$.
As an example, consider the sublevel set of the graph in Figure~\ref{fig:TwoTreesOne} before adding the edge with value $9.0$.
The quotient is a loop whose shadows are infinite polygonal lines that run parallel to each other in a diagonal direction, as shown in Figure~\ref{fig:shadow}.
To count the distinct shadows of $\CCC$, we introduce the \emph{periodicity lattice} of all vectors that keep the shadows invariant:
\begin{align}
  \Lambda_\CCC &= \{ u \in \Lambda \mid \ccc = \ccc + u \mbox{\rm ~for at least one and therefore every shadow~} \ccc \mbox{\rm ~of~} \CCC \}.
\end{align}
For example, translating the shadow in Figure~\ref{fig:shadow} along the vector $(1,1)$ maps the polygonal line to itself.
All other vectors for which this holds are integer multiples of this vector.
Hence, the periodicity lattice of the loop is the $1$-dimensional lattice spanned by $(1,1)$.
With this, counting the shadows of $\CCC$ reduces to counting the elements of $\Lambda / \Lambda_\CCC$.
If the cardinality of this quotient is finite, then this gives the number of shadows.
On the other hand, if the cardinality is infinite, we wish to count the elements that appear within a spherical window of radius $R$.
To this end, we introduce a monomial in the variable $R$, such that the coefficient gives the density and the exponent gives the growth-rate.
Write $\nu_q$ for the $q$-dimensional volume of the unit ball in $\Rspace^q$, so $\nu_0 = 1$, $\nu_1 = 2$, $\nu_2 = \pi$, $\nu_3 = \frac{4 \pi}{3}$, etc.
\begin{definition}[Shadow Monomial]
  \label{dfn:shadow_monomial}
  Let $\Lambda \subseteq \Rspace^d$ be a $d$-dimensional lattice, and $\filter \colon K \to \Rspace$ a $\Lambda$-periodic filter.
  Letting $\CCC \in \MTree (\filter/\Lambda)$ be a connected component of a sublevel set of the quotient filter, 
  and $p$ the dimension of the corresponding periodicity lattice, $\Lambda_\CCC$, the \emph{shadow monomial} of $\CCC$ is
  \begin{align}
    \Shadow{\CCC}{R} &= \frac{\volume{p}{\Lambda_\CCC}}{\volume{d}{\Lambda}} \cdot \nu_{d-p} R^{d-p} .
  \end{align}
\end{definition}
Up to lower-order terms, this monomial gives the number of elements of $\Lambda / \Lambda_\CCC$ that have points inside the sphere of radius $R$ centered at the origin in $\Rspace^d$.
For example, the loop whose shadows are the polygonal lines in Figure~\ref{fig:shadow}, we have $\volume{1}{\Lambda_\CCC} = \sqrt{2}$, $\volume{2}{\Lambda} = 1$, and $\nu_1 = 2$, so $\Shadow{\CCC}{R} = 2 \sqrt{2} R^1$, and up to a constant error term this is the number of polygonal lines that pass through a disk of radius $R$.
\begin{figure}[hbt]
  \centering \vspace{0.0in}
  \resizebox{!}{1.6in}{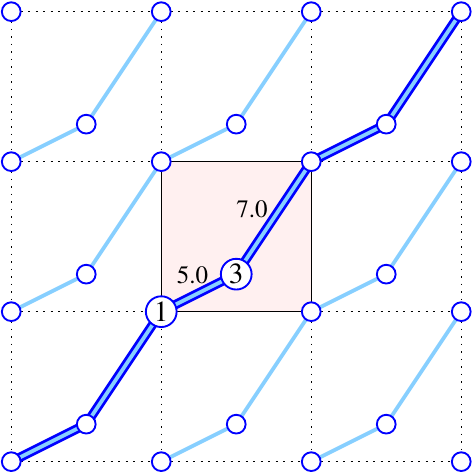}
  \caption{\footnotesize Each shadow of the loop in the quotient complex is an infinite polygonal line with periodicity lattice spanned by the vector $(1,1)$.
  The length of its unit cell is $\sqrt{2}$, which implies that its shadow monomial is $2 \sqrt{2} R$.}
  \label{fig:shadow}
\end{figure}
This property is fundamental to our results, so we give a formal statement and a proof.
\begin{lemma}[Counting Inside Sphere]
  \label{lem:counting_inside_sphere}
  Let $\Lambda \subseteq \Rspace^d$ be a $d$-dimensional lattice and $\Lambda_\CCC \subseteq \Lambda$ a $p$-dimensional sublattice.
  Then the number of elements of $\Lambda / \Lambda_\CCC$ that have a non-empty intersection with the ball of radius $R$ centered at $0 \in \Rspace^d$ is $\Shadow{\CCC}{R} + \bigOh{R^{d-p-1}}$.
\end{lemma}
\begin{proof}
  We begin with the two extreme cases.
  In the zero-dimensional case, when $p=0$, we have $\Lambda_\CCC = \{0\}$ and $\Lambda / \Lambda_\CCC = \Lambda$.
  We thus count the points of $\Lambda$ inside the $d$-ball of radius $R$ centered at $0$, which up to lower-order terms is the volume of the ball divided by the volume of the unit cell: $\nu_d R^d / \volume{d}{\Lambda} = \Shadow{\CCC}{R}$ because $\volume{0}{\Lambda_\CCC} = 1$.
  Observe that this estimate also holds if we count the points inside the $d$-ball with radius $R-1$.
  In the full-dimensional case, when $p=d$, we have $\dim{\Lambda_\CCC} = \dime{\Lambda}$, so $\Lambda / \Lambda_\CCC$ is finite, with cardinality $\volume{d}{\Lambda_\CCC} / \volume{d}{\Lambda}$.
  Indeed, any unit cell of $\Lambda_\CCC$ contains a point from each class in the quotient, so for a large enough constant $R$, the $d$-ball with radius $R$ contains a point from each class.
  Furthermore, $\volume{d}{\Lambda_\CCC} / \volume{d}{\Lambda} = \Shadow{\CCC}{R}$ because $\nu_0 R^0 = 1$.

  \smallskip
  In the general case, let $\Lambda'$ be the maximal sublattice of $\Lambda$ that satisfies $\Span{\Lambda'} = \Span{\Lambda_\CCC}$.
  By construction, $p = \dime{\Lambda'} = \dime{\Lambda_\CCC}$, so we can apply the result for the full-dimensional case, in which we count exactly $\volume{p}{\Lambda_\CCC} / \volume{p}{\Lambda'}$ translates of $\Lambda_\CCC$ within $\Lambda'$.
  Next project $\Lambda$ orthogonally onto the linear subspace orthogonal to $\Span{\Lambda'}$, which gives a $(d-p)$-dimensional lattice, $\Lambda''$.
  The unit cells satisfy $\volume{p}{\Lambda'} \cdot \volume{d-p}{\Lambda''} = \volume{d}{\Lambda}$.
  Using the result for the zero-dimensional case, we count $\nu_{d-p} R^{d-p} / \volume{d-p}{\Lambda''} + \bigOh{R^{d-p-1}}$ points of $\Lambda''$ within the ball of radius $R$, and multiplying the two counts while ignoring the lower order term, we get 
  \begin{align}
    \frac{\volume{p}{\Lambda_\CCC}}{\volume{p}{\Lambda'}} \cdot \nu_{d-p} R^{d-p} \frac{1}{\volume{d-p}{\Lambda''}}
    &= \frac{\volume{p}{\Lambda_\CCC}}{\volume{p}{\Lambda'}} \frac{\volume{p}{\Lambda'}}{\volume{d}{\Lambda}} \cdot \nu_{d-p} R^{d-p}
     = \Shadow{\CCC}{R} .
  \end{align}
  This expression measures the product of the $(d-p)$-ball with radius $R$ and the $p$-ball with constant radius.
  This product is not contained in the $d$-ball of radius $R$, but if we shrink the $(d-p)$-ball to radius $R-1$, then it is contained in the $d$-ball, assuming $R$ is sufficiently large.
  This changes the count only by a lower-order term, which implies the claim in the general case.
\end{proof}

\subsection{Periodic Merge Tree}
\label{sec:2.4}

To cope with the potentially infinitely many connected components of a periodic complex, we construct the merge tree for a finite representation in the torus and decorate this tree with a shadow monomial at every point.
Each shadow monomial is specified by a real coefficient and an integer exponent, and we write $\Monomial [R]$ for the set of such monomials.
\begin{definition}[Periodic Merge Tree]
  \label{dfn:periodic_merge_tree}
  Let $\Lambda \subseteq \Rspace^d$ be a $d$-dimensional lattice, and $\filter \colon K \to \Rspace$ a $\Lambda$-periodic filter.
  The \emph{periodic merge tree} of $\filter$ with respect to $\Lambda$, denoted $\perMTree{\filter}{\Lambda}$, is the merge tree of the quotient filter, $\MTree{} = \MTree{(\filter/\Lambda)} = \slSpace{K}{\filter/\Lambda}_\sim$, together with its height function, $\Height \colon \MTree{} \to \Rspace$, and the \emph{frequency function}, $\Frequency \colon \MTree{} \to \Monomial[R]$, which maps each point $\CCC \in \MTree$ to its shadow monomial.
\end{definition}
Whenever we talk about a periodic merge tree, we will tacitly assume that it comes with its height and frequency functions.
The notational ambiguity between merge trees and periodic merge trees is deliberate and emphasizes that every periodic merge tree is also a merge tree.
See Figure~\ref{fig:TwoTreesTwo} for two examples, in which we write the shadow monomials above the beams of the two periodic merge trees.
For notational reasons, we write $\Shadow{\CCC}{R} = \Frequency (\CCC) (R)$ when a value of the frequency function is applied to a radius, $R$.
Since connected components tend to grow
with increasing threshold, it is plausible that the shadow monomial can only get smaller when it changes as we move from left to right in the merge tree.
To make this intuition concrete, we write $t R^{d-q} < s R^{d-p}$ 
if either $d-q < d-p$ or $d-q = d-p$ and $t < s$, and we write $t R^{d-q} \leq s R^{d-p}$ if equality is allowed.
\begin{figure}[hbt]
  \centering \vspace{0.0in}
  \resizebox{!}{1.75in}{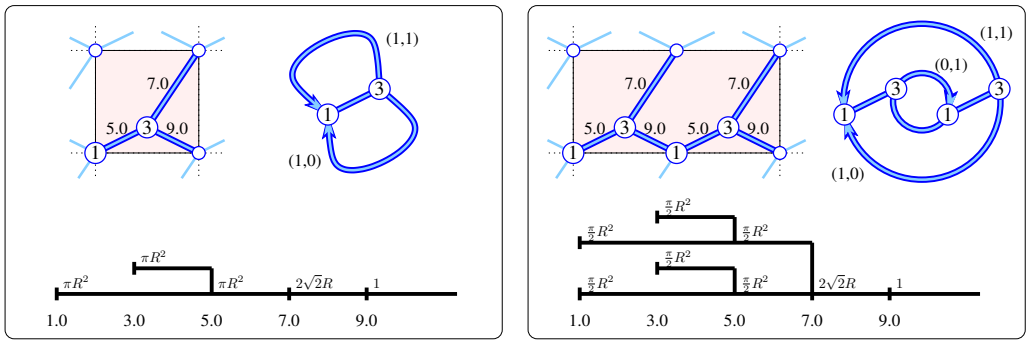}
  \caption{\footnotesize The same graphs and periodic merge trees as in Figure~\ref{fig:TwoTreesOne} but with additional information.
  Edges with non-zero shift vectors (to be defined in Section~\ref{sec:3.2})
  are drawn as (directed) arcs and labeled with their shift vectors, while edges with zero shift vectors remain undirected and without vector.
  In addition to the appearances at $t = 1.0, 3.0$ and the merger at $t = 5.0$, there are two catenations at $t = 7.0, 9.0$ that define the shadow monomials decorating the beams of the periodic merge trees in the \emph{left panel}.
  Note that the tree in the \emph{right panel} has twice as many subtrees rooted at the point labeled $7.0$, and that the shadow monomials compensate for the increased number of beams.
  Indeed, we have two events at each of the first four values defining the periodic merge tree, with a merger followed by a catenation at $t = 7.0$.}
  \label{fig:TwoTreesTwo}
\end{figure}
\begin{lemma}[Monotonicity]
  \label{lem:monotonicity}
  Let $\Lambda \subseteq \Rspace^d$ be a $d$-dimensional lattice, $\filter \colon K \to \Rspace$ a $\Lambda$-periodic filter, and 
  $\MTree = \perMTree{\filter}{\Lambda}$
  the periodic merge tree of $\filter$ with frequency function $\Frequency \colon \MTree{} \to \Monomial [R]$.
  If a point $\BBB \in \MTree{}$ covers another point $\AAA \in \MTree{}$, then $\Frequency(\BBB) \leq \Frequency(\AAA)$.
\end{lemma}
\begin{proof}
  Since $\BBB$ covers $\AAA$, the periodicity lattice of $\AAA$ is a sublattice of that of $\BBB$: $\Lambda_\AAA \subseteq \Lambda_\BBB$.
  Indeed, if $\aaa = \aaa+u$ for every shadow $\aaa$ of $\AAA$, then $\bbb = \bbb +u$ for every shadow $\bbb$ of $\BBB$, but not necessarily the other way round.
  The exponents of $\Frequency(\AAA)$ and $\Frequency(\BBB)$ are $d - \dime{\Lambda_\AAA}$ and $d - \dime{\Lambda_\BBB}$.
  If $\dime{\Lambda_\AAA} < \dime{\Lambda_\BBB}$, then $\Frequency(\BBB) < \Frequency(\AAA)$, and if $\dime{\Lambda_\AAA} = \dime{\Lambda_\BBB}$, then the coefficients determine the order, and the only difference between them is the volume of the periodicity lattice, which is at least as large for $\AAA$ as it is for $\BBB$.
  We thus get $\Frequency(\BBB) \leq \Frequency(\AAA)$ in general.
\end{proof}

Since the shadow monomial can only decrease along a left-to-right trajectory in the merge tree, its exponent progresses from $d$ down to $0$, possibly skipping some integers.
We therefore define the \emph{$(d-p)$-ary era} of a beam as the interval of points whose shadow monomials have exponent $d-p$.
For example, the longest beams of the periodic merge trees in Figure~\ref{fig:TwoTreesTwo} have three eras each, characterized by the exponent of $R$, which shrinks from left to right.
We further subdivide each era into \emph{epochs}, which are maximal intervals in which the coefficient of the shadow monomial stays constant.
From one epoch to the next on the same beam, the periodicity lattice grows to a superlattice, which implies a strengthening of Lemma~\ref{lem:monotonicity}, namely that the coefficient drops by division with an integer larger than $1$.
The graph in Figure~\ref{fig:TwoTreesTwo} is unfortunately too small to have eras consisting of more than one epoch, but see Section~\ref{sec:6} for a more elaborate example in $\Rspace^3$.

\section{The Algorithm}
\label{sec:3}

This section presents the algorithm we use to construct the periodic merge tree of a $\Lambda$-periodic filter, $\filter \colon K \to \Rspace$.
This tree is but the merge tree of the quotient filter, $\filter/\Lambda \colon K/\Lambda \to \Rspace$, equipped with the frequency function.
The construction of the merge tree has been amply studied prior to this paper; see e.g.\ \cite{SmMo20}, so we will focus on the frequency function.

\subsection{Critical Events}
\label{sec:3.1}

Similar to Kruskal's classic algorithm for minimum spanning trees \cite{Kru56}, we process the cells of $K / \Lambda$ in the order of the filter and update the periodic merge tree at each step.
The only relevant cells are the vertices and edges, and we distinguish the edges in the quotient complex whose endpoints lie in the same and in different connected components.
In the former case, we consider a shadow of the edge and further distinguish when its endpoints belong to the same or to different shadows of the connected component.
We therefore have three types of critical events that characterize the evolution of the connected components:
\smallskip \begin{itemize}
  \item \emph{appearance:} a vertex is added, which starts a new connected component or, equivalently, a new beam of the periodic merge tree;
   \item \emph{merger:} an edge with endpoints in two different connected components is added, which merges these components into one or, equivalently, ends a beam by joining it to another in the merge tree;
  \item \emph{catenation:} an edge with both endpoints in the same connected component is added and the endpoints of its shadow belong to two different shadows of that component, which enlarges the periodicity lattice of the component and shrinks its shadow monomial;
\end{itemize} \smallskip
see Figure~\ref{fig:TwoTreesTwo}, where the tree in the left panel experiences two appearances, one merger, and two catenations, while the tree in the right panel experiences four appearances, three mergers, and two catenations.
To add a new vertex in a appearance, we set its periodicity lattice to $\{0\}$.
To add a new edge, $e$, let $\AAA, \BBB$ be the connected components that contain its endpoints.
If $\AAA = \BBB$, then $e$ together with a path connecting its endpoints form a loop in the quotient complex, and the shadow of this loop can either be a loop in the periodic complex, or a path that connects different shadows of the same vertex.
Hence, the addition of the edge has either no effect or it enlarges the periodicity lattice of the component and thus decreases the shadow monomial. 
Writing $\Lambda_e$ for the periodicity lattice of the loop, and writing $\CCC$ for the component $\AAA$ after adding $e$,
the periodicity lattice of $\CCC$ is the smallest common superlattice: $\Lambda_\CCC = \Lambda_\AAA + \Lambda_e$.
On the other hand, if $\AAA \neq \BBB$, the addition of $e$ merges the two components, and we write $\CCC$ for the new component, which is the union of $\AAA$, $\BBB$, and $e$.
The periodicity lattice of $\CCC$ is again the smallest common superlattice:
$\Lambda_\CCC = \Lambda_\AAA + \Lambda_\BBB$.

\subsection{Periodicity Lattice and Loops}
\label{sec:3.2}

There is a direct relation between the periodicity lattice of a connected component and its loops.
Since our primary concern are components (and not cycles of dimension $1$ or higher), we may assume that a component, $\CCC$, is an undirected graph.
Nonetheless, we need to reason about loops in this graph, and this is done more conveniently in the representation of $\CCC$ as a directed graph, in which the (undirected) \emph{edge} connecting $x$ to $y$ is replaced by two (directed) \emph{arcs}, one from $x$ to $y$ and the other from $y$ to $x$.
\begin{definition}[Paths and Loops]
  \label{dfn:paths_and_loops}
  In a directed graph, a \emph{path} from $x$ to $y$ is a sequence of arcs, $a_i$ from $x_i$ to $x_{i+1}$, for $1 \leq i \leq k$, with $x_1=x$ and $x_{k+1}=y$.
  This path is a \emph{loop} if, in addition, $x = y$.
  A path or loop is \emph{simple} if no arc is repeated.
\end{definition}
In our context, each arc is labeled with an integer vector that records the relative position of the unit cells where a trajectory that is a shadow of the arc starts and ends.
Specifically, let $a$ be an arc from $x$ to $y$, write $x_0, y_0$ for the shadows of the endpoints inside a common unit cell, and let $x_0+u, y_0+w$ be the endpoints of a shadow of the arc.
Then the \emph{shift vector} of the arc is $\Shift{a} = U^{-1} (w-u)$.
Correspondingly, the shift vector of the arc from $y$ to $x$ is $\Shift{-a} = - \Shift{a}$.
The \emph{drift vector} of a path is the sum of shift vectors of its arcs, and similarly for a loop.
Assuming the directed graph represents a connected undirected graph, $\CCC$, we can designate one of its vertices as the \emph{root} and extend every loop so it starts and ends at the root without changing its drift vector.
Any two loops can therefore be concatenated while adding their drift vectors.
The collection of drift vectors thus satisfies the properties of a lattice: for every vector we also have its integer multiples, and for every two vectors we also have their sum.
Hence, this collection is isomorphic to a sublattice of $\Zspace^d$, and we write $V \subseteq \Zspace^d$ for its basis.
The periodicity lattice of the undirected graph, $\Lambda_\CCC$, has basis $U V$.
  
\smallskip \noindent \textsc{Example.}
Consider the right panel of Figure~\ref{fig:TwoTreesTwo}.
The edges of the quotient graph with non-zero shift vectors are drawn directed to avoid possible ambiguities.
The counterclockwise version of the \emph{upper loop} consists of the edges with filter values $5.0, 7.0, 5.0, 7.0$ 
and drift vector $v_1 = (0,0) + (0,1) + (0,0) + (1,1) = (1,2)$, which means this loop moves one unit cells to the right and two unit cells up.
The basis vectors that span the rectangular unit cell are $u_1 = (2.0, 0.0)$ and $u_2 = (0.0, 1.0)$, so this drift vector corresponds to
\begin{align}
  U \cdot v_1 &= \begin{bmatrix}
             2.0  &  0.0  \\  0.0  &  1.0
           \end{bmatrix} \cdot
           \begin{bmatrix} 1  \\  2 \end{bmatrix}
   =  \begin{bmatrix} 2.0 \\ 2.0 \end{bmatrix} ,
\end{align}
which is a vector in $\Lambda$.
Its length is $2 \sqrt{2}$, which explains the shadow monomial decorating the epoch that starts at $t = 7.0$ and ends at $t = 9.0$.
After adding the two edges with filter value $9.0$, we get additional loops:  the \emph{inner loop} with drift vector $v_2 = (0,0)-(0,1) = (0,-1)$, and the \emph{lower loop} with drift vector $v_3 = (0,0,)+(0,0)+(0,0)-(1,0) = (-1,0)$. 
All other loops are concatenations of these three loops, and thus their drift vectors are sums of these three drift vectors.
Writing $V$ for the matrix whose columns are $v_1, v_2, v_3$, the corresponding periodicity lattice is thus spanned by
\begin{align}
  U \cdot V &= \begin{bmatrix}
                 2.0  &  0.0  \\  0.0  &  1.0
               \end{bmatrix} \cdot
               \begin{bmatrix} 
                 1 & 0 & -1 \\ 2 & -1 & 0
               \end{bmatrix}
   =  \begin{bmatrix} 2.0 & 0.0 & -2.0 \\ 
                      2.0 &-1.0 & 0.0  
      \end{bmatrix} .
\end{align}
The last vector is redundant, so we may choose the first two as a basis.
They span a unit cell of area $2.0$, which is the same area as the unit cell of $\Lambda$ (the shaded rectangle).
This explains the shadow monomial decorating the last epoch in the periodic merge tree.

\subsection{Implementing with Union-Find}
\label{sec:3.3}

For the computation of the periodicity lattice, it suffices to look at the drift vectors of a basis of the loop space.
We work with a basis defined by a spanning tree of $\CCC$: each edge not in this tree belongs to a unique simple cycle in which all other edges are taken from the spanning tree.
Choosing a direction for this edge, we get a corresponding simple loop anchored at the root of $\CCC$, and the basis consists of one such loop for each edge not in the spanning tree.
It is convenient to use the spanning tree that consists of the edges corresponding to past merge events for this purpose.
The construction of this tree can be reduced to a sequence of \emph{find} and \emph{union} operations---the former are used to decide whether the two endpoints of an edge belong to the same or to different components, and the latter merge the two components that contain the endpoints, assuming they are different.

\smallskip
To put the theory into practice, we use a data structure that supports constant time find and amortized logarithmic time union operations.
Faster implementations of the \emph{union-find} data type exist; see e.g.\ the survey in \cite{GaIt91}, but they are not required for constructing the periodic merge tree within the desired time bound.
The data structure stores the vertices of $\CCC$ in a linked list whose first vertex is the root $r$ of $\CCC$.
We write $\Next{x}$ for the successor of $x$, which is \texttt{null} if $x$ is last.
In addition, $x$ stores a link to the root, $\Root{x} = r$, and the drift vector of the unique simple path from $r$ to $x$ inside the spanning tree, denoted $\Drift{x}$.
Finally, the root stores the number of vertices in the component, denoted $\Size{r}$, the vertex with smallest filter value in $\CCC$, denoted $\Oldest{r}$, as well as an integer matrix, $\Basis{r} = V$, such that $U V$ is a basis of the periodicity lattice of $\CCC$.

\smallskip
Given an edge with endpoints $x,y$, the find operation returns their respective roots: $r = \Root{x}$ and $s = \Root{y}$.
The vertices belong to the same component iff $r = s$, so we can distinguish between the cases in which the edge forms a loop or connects two components in constant time.
In the former case, we get a new path from $r$ to $y$, possibly with a different drift vector,
which requires updating the periodicity lattice.
Let $a$ be the directed version of the edge, leading from $x$ to $y$:
\begin{tabbing}
  mmm\=mm\=mm\=mm\=mm\=mm\=mm\=mm\=mm\=mm\=mm\=\kill
  {\footnotesize 01} \> \texttt{if} $r = s$ \texttt{then} \=
    $v = \Drift{x} + \Shift{a} - \Drift{y}$; \\*
  {\footnotesize 02} \> \> $\Basis{r} = \textsc{Reduce} (\Basis{r}, v)$ \\*
  {\footnotesize 03} \> \texttt{endif}.
\end{tabbing}
In the latter case, we form the union of the two components.
To do this efficiently, we pick the smaller component, update the information at all its vertices, and insert the updated list right after the root of the larger component.
In addition, we compute the periodicity lattice of the union:
\begin{tabbing}
  mmm\=mm\=mm\=mm\=mm\=mm\=mm\=mm\=mm\=mm\=mm\=\kill
  {\footnotesize 01} \> \texttt{if} $r \neq s$ \texttt{then} assume $\Size{s} \leq \Size{r}$; \\*
  {\footnotesize 02} \> \> $\Oldest{r} = \min \{ \Oldest{r}, \Oldest{s} \}$;
  $\Basis{r} = \textsc{Reduce} (\Basis{r}, \Basis{s})$; \\
  {\footnotesize 03} \> \> $v = \Drift{x} + \Shift{a} - \Drift{y}$; $z = s$; \\
  {\footnotesize 04} \> \> \texttt{while} $z \neq \texttt{null}$ \texttt{do} \\*
  {\footnotesize 05} \> \> \> $\Root{z} = r$; $\Drift{z} = v + \Drift{z}$;
    $last = z$; $z = \Next{z}$ \\*
  {\footnotesize 06} \> \> \texttt{endwhile}; \\
  {\footnotesize 07} \> \> $\Size{r} = \Size{r} + \Size{s}$;
    $\Next{last} = \Next{r}$; $\Next{r} = s$ \\*
  {\footnotesize 08} \> \texttt{endif.}
\end{tabbing}
The algorithm for updating the periodicity lattice, including its running time, will be discussed in the next subsection.
The running time of all other operations is easily analyzed.
The periodic merge tree of a $\Lambda$-periodic filter, $\filter \colon K \to \Rspace$, is constructed incrementally, by adding one vertex or edge at a time, and each addition gives rise to at most two find operations and at most one union operation.
Letting $n$ be the number of vertices of $K / \Lambda$, there are $n-1$ union operations in total.
Each vertex that receives a new drift vector ends up in a component at least twice the size of its old component.
Hence, every vertex changes its root and drift vector at most $\log_2 n$ times.
Letting $m$ be the number of edges of $K / \Lambda$, we observe that each find operation takes constant time, which implies that $\bigOh{n \log n + m}$ time suffices for this part of the algorithm.

\subsection{Euclid's Algorithm for Lattices}
\label{sec:3.4}

Next, we explain how the bases of two periodicity lattices can be processed to give the basis of their sum.
Given sublattices, $\Lambda', \Lambda'' \subseteq \Lambda$, the vectors in their bases span $\Lambda' + \Lambda''$, which is the smallest sublattice of $\Lambda$ that contains $\Lambda'$ and $\Lambda''$.
Since there are possibly more vectors than necessary, we reduce the matrix of basis vectors to its \emph{Hermite normal form}; see Schrijver~\cite{Sch98} for general background.
Working with the corresponding integer matrices, this form can be computed in time polynomial in the size of the matrix and its entries; see \cite{KaBa79} for the historically first such algorithm. 
To be self-contained, we give a simple algorithm that reduces the matrix by repeated application of Euclid's algorithm for the greatest common divisor (gcd) of two integers; see also \cite{HMM98}.
We focus on constant size matrices and sacrifice the polynomial running time for the simplicity of the algorithm.

\smallskip
Let $M$ be a $d \times c$ integer matrix, and interpret its columns as vectors that span a sublattice of $\Zspace^d$.
The algorithm reduces $M$ using only three types of column operations: multiply a column with $-1$, exchange two columns, and subtract one column from another.
An easy but important observation is that these operations preserve the lattice spanned by the columns.
We reduce $M$ from left to right such that the non-zero columns of the resulting lower triangular matrix define a basis of the lattice:
\begin{tabbing}
  mmm\=mm\=mm\=mm\=mm\=mm\=mm\=mm\=mm\=mm\=mm\=\kill
  {\footnotesize 01} \> $i = j = 1$; \\*
  {\footnotesize 02} \> \texttt{while} $i \leq d$ \texttt{do} \\*
  {\footnotesize 03} \> \> \texttt{if} $\exists \ell$ with $j \leq \ell \leq c$ \texttt{and} $M[i,\ell] \neq 0$ \texttt{then} assume $M[i,j] > 0$; \\*
  {\footnotesize 04} \> \> \> \texttt{for} $k=j+1$ \texttt{to} $c$ \texttt{do} assume $M[i,k] \geq 0$; \\
  {\footnotesize 05} \> \> \> \> \texttt{while} $M[i,j] > 0$ \texttt{and} $M[i,k] > 0$ \texttt{do} \\*
  {\footnotesize 06} \> \> \> \> \> subtract $\floor{{M[i,j]} / {M[i,k]}}$ times column $k$ from column $j$; \\*
  {\footnotesize 07} \> \> \> \> \> exchange columns $j$ and $k$ \\*
  {\footnotesize 08} \> \> \> \> \texttt{endwhile} \\* 
  {\footnotesize 09} \> \> \> \texttt{endfor}; \\
  {\footnotesize 10} \> \> \> \texttt{for} $k = 1$ \texttt{to} $j-1$ \texttt{do} \\*
  {\footnotesize 11} \> \> \> \> subtract $\floor{{M[i,k]} / {M[i,j]}}$ times column $j$ from column $k$ \\*
  {\footnotesize 12} \> \> \> \texttt{endfor}; $j = j+1$ \\*
  {\footnotesize 13} \> \> \texttt{endif}; $i = i+1$ \\*
  {\footnotesize 14} \> \texttt{endwhile.}
\end{tabbing}
The \texttt{while}-loop in lines~{\footnotesize 05} to {\footnotesize 08} effectively computes the gcd of the $i$-th entries of columns $j$ and $k$; that is: $M[i,j] = \GCD{M[i,j]}{M[i,k]}$ and $M[i,k] = 0$.
This is prepared by ascertaining $M[i,j] > 0$ and $M[i,k] \geq 0$ in lines~{\footnotesize 03} and {\footnotesize 04} by the possible exchange of columns $j$ and $\ell$ and by possible multiplication with $-1$.
After finishing the \texttt{for}-loop in lines~{\footnotesize 04} to {\footnotesize 09}, all entries that succeed $M[i,j]$ in its row are zero, and after finishing the \texttt{for}-loop in lines~{\footnotesize 10} to {\footnotesize 12}, all entries that precede $M[i,j]$ in its row are non-negative and smaller than $M[i,j]$.
After running the algorithm, each column has strictly more leading zeros than the preceding column.
Hence, there are at most $d$ non-zero columns left, and they form a basis of the lattice defined by $M$.
Indeed, the matrix is in Hermite normal form, as defined in \cite{Mad00} for the not necessarily full-rank case.

\smallskip
The running time depends, among other things, on the entries in $M$.
To study this dependence, we introduce the \emph{magnitude} of $M$ as the maximum absolute entry in the matrix, denoted $\norm{M}_\infty$.
Similarly, we talk about the magnitude of an integer vector and the magnitude of a basis of integer vectors.
Setting $N = \norm{M}_\infty$, we show that the algorithm reduces $M$ in time at most logarithmic in $N$.
Note however that there are faster variants of this algorithm; see e.g.\ \cite{HMM98,Sto00}.
 
\smallskip
A key ingredient of the proof is B\'{e}zout's identity, which is a statement about Euclid's algorithm for computing the gcd of positive integers $s \leq t$.
Specifically, it asserts that there are integers $x$ and $y$ such that $xs+yt = \GCD{s}{t}$ and $|x|, |y| \leq t$.
These coefficients are computed by the \emph{extended Euclidean algorithm} from the quotients that appear during the execution of Euclid's algorithm.
The latter substitutes $0$ and $\GCD{s}{t}$ for $s$ and $t$, so it is not surprising that there are also integers $x_0$ and $y_0$ such that $x_0 s + y_0 t = 0$ and $|x_0|, |y_0| \leq t$, and they can similarly be computed from the quotients in Euclid's algorithm.
In summary, the actions of Euclid's algorithm are equivalent to substituting $xs+yx$ and $x_0s+y_0t$ for $s$ and $t$, and the magnitude of the four coefficients does not exceed $t$.
\begin{lemma}[Time for Reduction]
  \label{lem:time_for_reduction}
  Given a lattice spanned by the columns of a $d \times c$ matrix of integer entries with absolute values at most $N$, the reduction algorithm takes time at most $\bigOh{d^2 c^d \log (2N)}$ to construct a basis of the lattice.
\end{lemma}
\begin{proof}
  The algorithm consists of four loops, of which the outer \texttt{while}-loop iterates at most $d$ times, and the two \texttt{for}-loops together iterate $c-1$ times.
  To bound the number of iterations of the inner \texttt{while}-loop, let $\multiplicity_i$ be the maximum absolute entry that appears in row $i$ at any time during the algorithm.
  We have $\multiplicity_1 \leq N$, because the reduction of the first row does not increase its entries.
  However, it may increase entries in the other rows.
  To understand by how much, we consider the side effect of Euclid's algorithm computing the gcd of $s = M[1,j]$ and $t = M[1,k]$, with $s \leq t$ or $t \leq s$.
  It substitutes $\GCD{s}{t} = xs+yt$ for $s$ in $M[1,j]$ and $0 = x_0s+y_0t$ for $t$ in $M[1,k]$.
  Since the reduction works on columns, it also substitutes $x M[2,j] + y M[2,k]$ for $M[2,j]$ and $x_0 M[2,j] + y_0 M[2,k]$ for $M[2,k]$.
  We have $|x|, |y|, |x_0|, |y_0| \leq N$, so the first gcd computation in row $1$ increases the magnitude of row $2$ from at most $N$ to at most $2N^2$.
  Similarly, every additional gcd computation in row $1$, increases the magnitude of row $2$ by at most a factor $2N$, so after $c-1$ such computations, we have $\multiplicity_2 \leq N (2N)^{c-1} = \tfrac{1}{2} (2N)^c$.
  We iterate this argument and get $\multiplicity_i \leq \tfrac{1}{2} (2 \multiplicity_{i-1})^c \leq \tfrac{1}{2} (2N)^{c^{i-1}}$ for $2 \leq i \leq d$.
  Taking the binary logarithm of the final and largest upper bound, we get
  \begin{align}
    \log_2 \multiplicity_d
      &\leq \log_2 \left( \tfrac{1}{2} (2N)^{c^{d-1}} \right)
       \leq c^{d-1} \log_2 (2N) .
  \end{align}
  Since Euclid's algorithm takes at most logarithmically many steps to compute the gcd of two integers, this is an upper bound on the number of iterations of the inner \texttt{while}-loop.
  Multiplying with $d (c-1)$, we get $d c^d \log_2 (2N)$ as an upper bound on the total number of iterations.
  Each iteration takes $\bigOh{d}$ time to subtract an integer multiple of a column from another, or to exchange two columns, which implies the claimed time bound.
\end{proof}

\noindent \emph{Remark on magnitude.}
While the above bound on the size of temporary entries in the matrix is rather high, it is known that the magnitude of the matrix in Hermite normal form is bounded from above by $(\sqrt{d} N)^d$ \cite[pages 91, 92]{Sto00}.
Since the matrix in this form is unique \cite[Theorem~1.5.2]{Mad00}, it follows that this upper bound also applies to the reduced matrix computed by the above algorithm.

\subsection{Magnitude and Time for Construction}
\label{sec:3.5}

The size of the entries in the matrices plays an important role in the running time for constructing a periodic merge tree.
We thus express $N$ in terms of the size of $K / \Lambda$.
Considering a (directed) edge of the quotient complex, $a$, we recall that $\Shift{a}$ records the relative position of the endpoints of a shadow.
The magnitude of this vector is its largest absolute entry, and we write $D$ for the maximum magnitude over all shift vectors of edges in $K / \Lambda$, noting that it depends on the chosen basis of $\Lambda$.

\smallskip
It is not difficult to see that there are cases in which an unfortunate choice of basis can cause an arbitrarily large $D$, even independent of the number of vertices and edges in $K / \Lambda$.
On the other hand, if $K$ is a Delaunay triangulation and the basis vectors are pairwise orthogonal, then $D \leq 1$.
This follows from \cite[Theorem 3.1]{DoHu97} but can also be proved directly, as illustrated in Figure~\ref{fig:hyperrectangle}.
Let $x$ and $y$ be the endpoints of an edge in the Delaunay triangulation, and $u_i$ a basis vector.
The four points, $x, x+u_i, y, y - u_i$ are the vertices of a parallelogram, and the angles at $x+u_i$ and $y-u_i$ are equal and at most $90^\circ$.
Indeed, if they exceed $90^\circ$, then every $(d-1)$-sphere that passes through $x$ and $y$ encloses at least one of $x+u_i$ and $y-u_i$, which contradicts that $x$ and $y$ are the endpoints of an edge in the Delaunay triangulation.
Since this holds for all basis vectors and their negatives, $\pm u_i$ for $1 \leq i \leq d$, we conclude that $y$ is contained in the hyper-rectangle with vertices $x \pm u_1 \pm u_2 \pm \ldots \pm u_d$.
But this hyper-rectangle overlaps only the unit cell that contains $x$ and the $3^d - 1$ neighboring unit cells.
It follows that $x$ and $y$ lie in neighboring unit cells, so $D \leq 1$, as claimed.
\begin{figure}[hbt]
  \centering \vspace{0.0in}
  \resizebox{!}{1.5in}{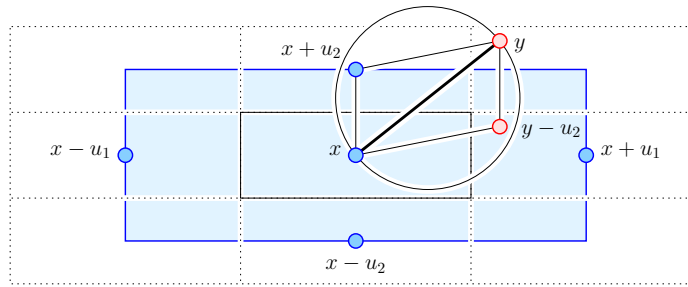}
  \vspace{-0.1in}
  \caption{\footnotesize The edge connecting $x$ to $y$ cannot be in the Delaunay triangulation if $y$ lies outside the hyper-rectangle whose facets are centered at the points $x \pm u_i$.
  This hyper-rectangle has the volume of $2^d$ unit cells and overlaps $3^d$ of them, which for the displayed $2$-dimensional case are drawn with \emph{dotted lines}.}
  \label{fig:hyperrectangle}
\end{figure}

\begin{lemma}[Magnitude of Basis]
  \label{lem:magnitude_of_basis}
  Let $\Lambda \subseteq \Rspace^d$ be a $d$-dimensional lattice with given matrix of basis vectors, $U$, let $\filter \colon K \to \Rspace$ be a $\Lambda$-periodic filter, let $m$ be the number of edges in $K / \Lambda$, and let $D$ the be maximum magnitude of their shift vectors.
  For every 
  $\CCC \in \perMTree{\filter}{\Lambda}$,
  let $V$ be the Hermite normal form of the lattice $U^{-1}(\Lambda_{\CCC}) \subseteq \Zspace^d$. Then, the basis, $U V$, of the periodicity lattice of $\CCC$ satisfies $\norm{V}_\infty \leq (\sqrt{d} Dm)^d$.
\end{lemma}
\begin{proof}
  Take a spanning tree of $\CCC$, and for each edge not in this tree, take the drift vector of the simple loop for which this edge is the sole edge not in the tree.
  There are at most $m$ edges in this loop, so its drift vector has magnitude at most $Dm$.
  We get fewer than $m$ such vectors, which define the lattice of the periodicity lattice of $\CCC$.
  Letting $M$ be the matrix whose columns are these drift vectors, we compute a basis by reducing $M$ to Hermite normal form.
  Since $\norm{M{_\infty}} \leq Dm$, this implies that the magnitude of the reduced matrix is at most $(\sqrt{d} Dm)^d$.
\end{proof}

The overall algorithm constructs a periodic merge tree by incrementally adding the vertices and edges of $K / \Lambda$.
Whenever an edge connects two components (a merge event) or forms a loop (a possible catenation), we use the reduction algorithm to compute a basis from two bases.
By Lemma~\ref{lem:magnitude_of_basis}, the magnitude of the input matrices is at most $(\sqrt{d} Dm)^d$.
Indeed, as the Hermite normal form of a lattice is unique, the iteratively computed Hermite normal form from our algorithm agrees with the Hermite normal form from Lemma~\ref{lem:magnitude_of_basis}.
\begin{theorem}[Time for Construction]
  \label{thm:time_for_construction}
  Let $\Lambda \subseteq \Rspace^d$ be a $d$-dimensional lattice with given basis, $\filter \colon K \to \Rspace$ a $\Lambda$-periodic filter, $n$ and $m$ the number of vertices and edges in the quotient complex, $K / \Lambda$, and $D$ the maximum magnitude of the shift vectors.
  Then the periodic merge tree of $\filter$ can be constructed in time $\bigOh{n \log n + m 2^d d^{d+3} \log (\sqrt{d} Dm)}$.
\end{theorem}
\begin{proof}
  The union-find data structure is maintained in constant time per edge and amortized logarithmic time per vertex, so $\bigOh{n \log n + m}$ time in total.
  The reduction algorithm is used at most once for each edge, and each time reduces an integer matrix with $d$ rows and at most $2d$ columns.
  By Lemma~\ref{lem:magnitude_of_basis}, its magnitude is at most $(\sqrt{d} Dm)^d$, and by Lemma~\ref{lem:time_for_reduction}, the time to reduce the matrix is at most $\bigOh{d^2 (2d)^d \log (2(\sqrt{d} Dm)^d)}$.
  Assuming $(\sqrt{d} D m)^d \geq 2$, we can simplify the bound to $\bigOh{d^3 (2d)^d \log (\sqrt{d}Dm)}$.
  After multiplying with the number of edges, we get $\bigOh{m d^3 (2d)^d \log (\sqrt{d} Dm)}$ time in total, which implies the claimed bound.
\end{proof}

In the important application of our algorithm to crystalline materials, the dimension is a constant, $d=3$.
Using the filtration of alpha complexes in the torus, it is typical that each vertex (atom center) is connected to at most some constant number of other vertices, so the number of edges is $m = \bigOh{n}$.
Similarly typical is that entries in the shift vectors are bounded by a small constant, $D$.
In this case, the algorithm constructs the periodic merge tree in time at most $\bigOh{n \log n}$.

\section{Properties}
\label{sec:4}

This section studies the invariance under different choices of lattices and the stability with respect to perturbations.
We begin with the introduction of a notion of distance between periodic merge trees, and follow up with an equivalence relation among periodic merge trees and the corresponding quotient pseudo-metric.

\subsection{Interleaving Distance}
\label{sec:4.1}

We adapt the interleaving distance between ordinary merge trees introduced in \cite{MBW13} to periodic merge trees by including the frequency function, which maps each point of the tree to its shadow monomial.
Recall that a point $\BBB \in \MTree$ covers $\AAA \in \MTree$ if $\Height(\AAA) \leq \Height(\BBB)$ and $\AAA \subseteq \BBB$.
\begin{definition}[Interleaving Distance]
  \label{dfn:interleaving_distance}
  Let $\MTree$ and $\MTree'$ be two periodic merge trees, with height functions $\Height \colon \MTree \to \Rspace$, $\Height' \colon \MTree' \to \Rspace$ and frequency functions $\Frequency \colon \MTree \to \Monomial [R]$, $\Frequency' \colon \MTree' \to \Monomial [R]$.
  For $\ee \geq 0$, continuous maps, $\varphi \colon \MTree \to \MTree'$ and $\psi \colon \MTree' \to \MTree$ are \emph{$\ee$-compatible} if for all $\CCC \in \MTree$, $\CCC' \in \MTree'$,
  \smallskip \begin{enumerate}[(i)]
    \item $\Height' \circ \varphi (\CCC) = \Height(\CCC) + \ee$ and $\Height \circ \psi (\CCC') = \Height'(\CCC') + \ee$;
    \item $\psi \circ \varphi (\CCC)$ covers $\CCC$ and $\varphi \circ \psi (\CCC')$ covers $\CCC'$;
    \item $\Frequency' \circ \varphi (\CCC) \leq \Frequency (\CCC)$ and $\Frequency \circ \psi (\CCC') \leq \Frequency' (\CCC')$.
  \end{enumerate} \smallskip
  The \emph{interleaving distance} between $\MTree$ and $\MTree'$, denoted $\Idist{\MTree}{\MTree'}$, is the infimum of the $\ee \geq 0$ for which there exist $\ee$-compatible maps from $\MTree$ to $\MTree'$ and back.
\end{definition}
We focus on \emph{finite} periodic merge trees, for which the infimum in Definition~\ref{dfn:interleaving_distance} is a minimum.
Throughout this section, we will therefore assume finite trees, with the benefit of simpler arguments involving the interleaving distance.

\smallskip
As an example, consider the two periodic merge trees in Figure~\ref{fig:TwoTreesTwo}.
Writing $\MTree$ for the tree in the left panel and $\MTree'$ for the tree in the right panel,
we need continuous maps $\varphi$ and $\psi$ whose compositions increase the height of each point by $2\ee$ in its own tree (i.e.\ move the point to the right in our drawing).
For small $\ee$, this is not possible because $\MTree'$ has two subtrees rooted at the point with height $7.0$, while $\MTree$ has only one such subtree.
The minimum height of any point is $1.0$, and since $\ee$ is at least half the height difference, we get $\ee \geq 3.0$.
We get another constraint from Condition {\sf (iii)}: the shadow monomial cannot increase from a point to its image.
All points with height less than $7.0$ in $\MTree$ have shadow monomials larger than those of the points with height $1.0$ in $\MTree'$, hence $\ee \geq 6.0$.
Indeed, for $\ee = 6.0$ we get maps $\varphi$ and $\psi$ that are $\ee$-compatible, which implies $\Idist{\MTree}{\MTree'} = 6.0$.

\smallskip
We extend the proof of the interleaving distance being a metric for merge trees \cite{MBW13} to periodic merge trees, but note that \cite{MBW13} missed the verification of positivity, which we add in our extension.
\begin{lemma}[Interleaving Distance is a Metric]
  \label{lem:interleaving_distance_is_metric}
  Let $\MTree, \MTree', \MTree''$ be periodic merge trees.
  Then
  \begin{itemize}
    \item $\Idist{\MTree}{\MTree'} \geq 0$ and $\Idist{\MTree}{\MTree'} > 0$ iff $\MTree \neq \MTree'$ (positivity);
    \item $\Idist{\MTree}{\MTree'} = \Idist{\MTree'}{\MTree}$ (symmetry);
    \item $\Idist{\MTree}{\MTree'} + \Idist{\MTree'}{\MTree''} \geq \Idist{\MTree}{\MTree''}$ (triangle inequality).
  \end{itemize}
\end{lemma}
\begin{proof}
  \emph{Positivity}:
  Clearly, $\Idist{\MTree}{\MTree'} = 0$ if $\MTree = \MTree'$.
  It remains to show $\Idist{\MTree}{\MTree'} > 0$ if $\MTree \neq \MTree'$.
  Recall that a complex in the $d$-dimensional torus is necessarily finite, so there are only finitely many critical events in the construction of the corresponding periodic merge tree.
  Let $t_0 < t_1 < \ldots < t_k$ be the values of the critical events that arise in the construction of $\MTree$ and $\MTree'$.
  They decompose both trees into a finite collection of intervals, each connecting points at consecutive critical values.
  There is an adjacency preserving bijection between the two collections of intervals that preserves heights and shadow monomials iff $\MTree = \MTree'$.
  But if there is no such bijection, then there are also no $\ee$-compatible maps for $\ee$ smaller than the minimum difference between consecutive critical values.
  Hence, $\Idist{\MTree}{\MTree'} > 0$.

  \smallskip \noindent
  \emph{Symmetry}:
  Since we can switch $\varphi$ and $\psi$ in Definition~\ref{dfn:interleaving_distance}, we have $\Idist{\MTree}{\MTree'} \leq \ee$ iff $\Idist{\MTree'}{\MTree} \leq \ee$, for every $\ee \geq 0$.
  Hence, $\Idist{\MTree}{\MTree'} = \Idist{\MTree'}{\MTree}$.

  \smallskip \noindent
  \emph{Triangle inequality}:
  Writing $\ee = \Idist{\MTree}{\MTree'}$ and $\ee' = \Idist{\MTree'}{\MTree''}$, there exist $\ee$-compatible maps $\varphi \colon \MTree \to \MTree'$, $\psi \colon \MTree' \to \MTree$ and $\ee'$-compatible maps $\varphi' \colon \MTree' \to \MTree''$, $\psi' \colon \MTree'' \to \MTree'$.
  It is easy to see that $\varphi' \circ \varphi \colon \MTree \to \MTree''$ and $\psi \circ \psi' \colon \MTree'' \to \MTree$ satisfy {\sf (i)} and {\sf (iii)} of Definition~\ref{dfn:interleaving_distance} for $\ee + \ee'$.
  To see {\sf (iii)}, we continuously advance $\CCC$ to the point $\Pi \in \MTree$ at height $\Height (\Pi) = \Height (\CCC) + 2 \ee'$.
  By continuity of $\varphi$, and because there is only one point of $\MTree'$ whose height is $\Height' (\varphi (\CCC)) + 2\ee'$ that covers $\varphi (\CCC)$, this point must be $\varphi (\Pi) = \psi' \circ \varphi' \circ \varphi (\CCC)$.
  Hence, $\psi \circ \varphi (\Pi) = \psi \circ \psi' \circ \varphi' \circ \varphi (\CCC)$.
  Since $\psi \circ \varphi (\Pi)$ covers $\Pi$ and therefore $\CCC$, this implies that $\varphi' \circ \varphi$ and $\psi \circ \psi'$ satisfy {\sf (ii)}
  for $\MTree$ and similarly 
  for $\MTree''$, which implies $\Idist{\MTree}{\MTree''} \leq \ee+\ee'$.
\end{proof}

\subsection{Splintering Periodic Merge Trees}
\label{sec:4.2}

While the periodic merge tree is independent of the choice of basis of a given lattice, it heavily depends on the lattice used for the filter; see Figure~\ref{fig:TwoTreesTwo}, where we get interleaving distance $\ee = 6.0$ just because we switch from the integer lattice on the left to a sublattice that drops every other column of integer points on the right.
We desire a notion of distance that tolerates different lattices, provided the filter is periodic with respect to both.
To this end, we introduce an equivalence relation and work with the corresponding quotient of the interleaving distance.
For each point $\CCC \in \MTree$, write $\MTree_\CCC$ for the subtree with topmost point $\CCC$; that is: $\MTree_\CCC$ consists of all points $\AAA \in \MTree$ covered by $\CCC$.
\begin{definition}[Equivalence by Splintering]
  \label{dfn:equivalence_by_splintering}
  Let $\MTree$ and $\MTree'$ be two periodic merge trees with height functions $\Height \colon \MTree \to \Rspace$, $\Height' \colon \MTree' \to \Rspace$ and frequency functions $\Frequency \colon \MTree \to \Monomial [R]$, $\Frequency' \colon \MTree' \to \Monomial [R]$.
  We say $\MTree'$ \emph{splinters} $\MTree$ if there exists a continuous surjection $\omega \colon \MTree' \to \MTree$ such that
  \smallskip \begin{enumerate}[(i)]
    \item the surjection preserves height: $\Height \circ \omega = \Height'$;
    \item it splits subtrees evenly: for every $\CCC \in \MTree$ and all $\AAA, \BBB \in \omega^{-1} (\CCC)$, the subtrees at $\AAA$ and $\BBB$ are equal, $\Idist{\MTree_\AAA'}{\MTree_\BBB'} = 0$,
    $\omega$ maps $\MTree_\AAA'$ as well as $\MTree_\BBB'$ surjectively onto $\MTree_\CCC$,
    and the shadow monomials are $\Frequency' (\AAA) = \Frequency'(\BBB) = \Frequency (\CCC) / \card{\omega^{-1} (\CCC)}$.
  \end{enumerate} \smallskip
  We call two periodic merge trees \emph{equivalent}, denoted $\MTree \simeq \NTree$, if there exists a periodic merge tree that splinters both, and we write $[\MTree]$ for the class of periodic merge trees equivalent to $\MTree$.
\end{definition}
To be certain that ``$\simeq$'' is indeed an equivalence relation, we need to verify that it is transitive.
First note that splintering is transitive: if $\MTree'$ splinters $\MTree$ and $\MTree''$ splinters $\MTree'$, then $\MTree''$ also splinters $\MTree$.
Second observe that two splinterings of the same tree can be composed: if $\MTree'$ and $\NTree'$ both splinter $\NTree$, then there exists $\MTree''$ that splinters both $\MTree'$ and $\NTree'$.
This operation acts like point-wise multiplication: if $\CCC \in \NTree$ has $i$ preimages in $\MTree'$ and $j$ preimages in $\NTree'$, then it has $ij$ preimages in $\MTree''$.
To see that this is well defined, we traverse the points of the two trees in parallel and in the order of decreasing height, making sure that the split is even whenever we encounter a point of bifurcation.
With this we argue the transitivity of ``$\simeq$'': assuming $\MTree \simeq \NTree \simeq \OTree$, there exist trees $\MTree'$ and $\NTree'$ splintering $\MTree, \NTree$ and $\NTree, \OTree$, respectively.
Since $\MTree'$ and $\NTree'$ both splinter $\NTree$, there exists $\MTree''$ that splinters $\MTree', \NTree'$.
By transitivity of splintering, $\MTree''$ also splinters $\MTree$ as well as $\OTree$, so $\MTree \simeq \OTree$ as required.

\smallskip
Splintering happens, for example, when we construct the periodic merge tree of a $\Lambda$-periodic filter for a sublattice of $\Lambda$; see Figure~\ref{fig:TwoTreesTwo}.
By combining Definitions~\ref{dfn:interleaving_distance} and \ref{dfn:equivalence_by_splintering}, we obtain a pseudo-distance function that considers equivalent such trees the same.
For this, we use the quotient pseudo-metric of a metric space with respect to an equivalence relation \cite[Definition~3.1.12]{BBI01}.
The proof that this yields a pseudo-metric is both standard and trivial. 
\begin{definition}[Interleaving Pseudo-distance]
  \label{dfn:interleaving_pseudo-distance}
  Let $[\MTree]$ and $[\MTree']$ be two equivalence classes of periodic merge trees.
  The \emph{interleaving pseudo-distance} between them is the infimum, over all sequences of equivalence classes and two trees per class, of the sum of interleaving distances:
  \begin{align}
    \Jdist{[\MTree]}{[\MTree']} &= \inf \left\{ \sum\nolimits_{i=1}^{k} \Idist{\MTree_i}{\MTree_i'} \right\} ,
  \end{align}
  in which $\MTree_1 \in [\MTree]$, $\MTree_k' \in [\MTree']$, and $\MTree_i' \simeq \MTree_{i+1} $ for all $1 \leq i \leq k-1$.
\end{definition}
As an example, consider again the two periodic merge trees in Figure~\ref{fig:TwoTreesTwo}.
Writing $\MTree$ and $\MTree'$ for the tree in the left and right panels, respectively, we recall that the interleaving distance is $\Idist{\MTree}{\MTree'} = 6.0$.
The interleaving pseudo-distance is however zero, because $\MTree'$ splinters $\MTree$ and thus they are in the same equivalence class. Indeed, we get $\MTree'$ by duplicating the subtree of $\MTree$ below the point at height $7.0$ and assigning half the shadow monomial to each point of the two copies of the subtree.

\subsection{Invariance of Equivalence Classes}
\label{sec:4.3}

As mentioned earlier, splintering occurs if we enlarge the unit cell by substituting a sublattice for the original lattice in the construction of the periodic merge tree.
We formalize this claim and prove it.
\begin{lemma}[Splintering from Sublattice]
  \label{lem:splintering_from_sublattice}
  Let $\Lambda \subseteq \Rspace^d$ be a $d$-dimensional lattice, $\filter \colon K \to \Rspace$ a $\Lambda$-periodic filter, and $\Lambda' \subseteq \Lambda$ a $d$-dimensional sublattice.
  Then $\perMTree{\filter}{\Lambda'}$ splinters $\perMTree{\filter}{\Lambda}$.
\end{lemma}
\begin{proof}
  To show that 
  $\MTree' = \perMTree{\filter}{\Lambda'}$ splinters  $\MTree = \perMTree{\filter}{\Lambda}$,
  we study the projection $\phi \colon \Rspace^d / \Lambda' \to \Rspace^d / \Lambda$ defined by mapping every point $x + \Lambda'$ to $x + \Lambda$.
  Since $\dime{\Lambda'} = \dime{\Lambda}$, the degree of the map is finite, namely $\degree{\phi} = \volume{d}{\Lambda} / \volume{d}{\Lambda}$, and the image of $\phi$ covers every point in the $d$-dimensional torus $\degree{\phi}$ times.
  The projection induces a surjection, $\omega \colon \MTree' \to \MTree$, that maps every point $\CCC' \in \MTree'$ to the point $\CCC = \omega (\CCC') \in \MTree$ with $\Height(\CCC) = \Height'(\CCC')$ and $\phi (\CCC') = \CCC$, where we note that $\CCC$ and $\CCC'$ are connected components of sublevel sets.
  The continuity of $\omega$ follows from the continuity of $\phi$ and the definition of quotient topology as the final topology with respect to the quotient map.
  By construction, $\omega$ preserves height, so it satisfies Condition~{\sf (i)} of Definition~\ref{dfn:equivalence_by_splintering}.

  \smallskip
  To prove Condition~{\sf (ii)}, we note that all preimages of $\CCC$ are translates of each other and thus have the same periodicity lattice and shadow monomials. 
  Similarly, their subtrees are identical. 
  To show that $\omega$ maps subtrees surjectively to the appropriate subtrees, recall that $\phi (\CCC') = \CCC$. Hence, for every component $\AAA$ in the subtree at $\CCC$, we have $\phi|_{\CCC'}^{-1} (\AAA) \neq \emptyset$, and hence at least one component at height $\Height(\AAA)$ in the subtree at $\CCC'$ that $\omega$ maps to $\AAA$.
  The only part of Condition~{\sf (ii)} left to prove is the shadow monomial, $\Frequency(\CCC')$. 
  To this end, note that the cardinality of the preimage, $\card{\omega^{-1} (\CCC)}$, depends on $\degree{\phi}$ but also on the periodicity lattices of $\CCC$ and $\CCC'$.
  By construction, these two lattices have the same dimension, $p$.
  The number of times the projection of $\CCC'$ covers $\CCC$ is thus finite, namely $\volume{p}{\Lambda_{\CCC'}} / \volume{p}{\Lambda_{\CCC}}$.
  Therefore, out of the $\degree{\phi}$ preimages of every point in $\Gamma$, $\volume{p}{\Lambda_{\CCC'}} / \volume{p}{\Lambda_{\CCC}}$ stem from the same component, so the number of different components is
  \begin{align}
    \card{\omega^{-1} (\CCC)}
       = \degree{\phi} \cdot \frac{\volume{p}{\Lambda_\CCC}}{{\volume{p}{\Lambda_{\CCC'}}}} .
  \end{align}
  This yields the following calculation for the shadow monomial:
  \begin{align}
    \Frequency' (\CCC') &= \frac{\volume{p}{\Lambda_{\CCC'}}}{\volume{d}{\Lambda'}} \cdot \nu_{d-p} R^{d-p} \\
      &= \frac{\volume{p}{\Lambda_{\CCC}}}{\volume{d}{\Lambda}} \cdot \nu_{d-p} R^{d-p} / \left[ \frac{\volume{d}{\Lambda'}}{\volume{d}{\Lambda}} \cdot \frac{\volume{p}{\Lambda_\CCC}}{\volume{p}{\Lambda_{\CCC'}}} \right] = \frac{\Frequency (\CCC)}{\card{\omega^{-1} (\CCC)}} ,
    \label{eqn:frequency2}
  \end{align}
  in which the left-hand side of \eqref{eqn:frequency2} is obtained using trivial substitutions, and the right-hand side follows because the expression in squared brackets is $1 / \card{\omega^{-1} (\CCC)}$.
  We conclude that $\omega$ satisfies all conditions of  Definition~\ref{dfn:equivalence_by_splintering}.
\end{proof}

This implies the first main property of periodic merge trees we set out to prove in this section: that their equivalence classes are invariant under different lattices and different bases of the same lattice used in the construction of the trees. 
\begin{theorem}[Invariance of Equivalence Classes of Periodic Merge Trees]
  \label{thm:invariance_of_periodic_merge_trees}
  Let $\Lambda \subseteq \Rspace^d$ be a $d$-dimensional lattice, $\filter \colon K \to \Rspace$ a $\Lambda$-periodic filter, and $\Lambda', \Lambda'' \subseteq \Lambda$ two $d$-dimensional sublattices.
  Then 
  $\perMTree{\filter}{\Lambda'}$ and $\perMTree{\filter}{\Lambda''}$
  are equivalent and thus the interleaving pseudo-distance between their equivalence classes vanishes: $\Jdist{[\perMTree{\filter}{\Lambda'}]}{[\perMTree{\filter}{\Lambda''}]} = 0$.
\end{theorem}
\begin{proof}
  We show that the two trees are equivalent by finding a common splintering.
  For this, consider the lattice $\Lambda'''=\Lambda' \cap \Lambda''$. 
  It is a sublattice of both $\Lambda'$ and $\Lambda''$.
  By Lemma~\ref{lem:splintering_from_sublattice}, 
  $\perMTree{\filter}{\Lambda'''}$ splinters both, $\perMTree{\filter}{\Lambda'}$ and $\perMTree{\filter}{\Lambda''}$,
  which implies that the interleaving pseudo-distance vanishes.
\end{proof}

\subsection{Stability of Periodic Merge Trees}
\label{sec:4.4}

The second property addressed in this section is the stability of periodic merge trees under perturbations of the filter.
More specifically, we show that the interleaving pseudo-distance between the equivalence classes of two periodic merge trees is bounded from above by the $L_\infty$-distance between the two filters.
\begin{figure}[hbt]
  \centering \vspace{0.0in}
  \resizebox{!}{0.9in}{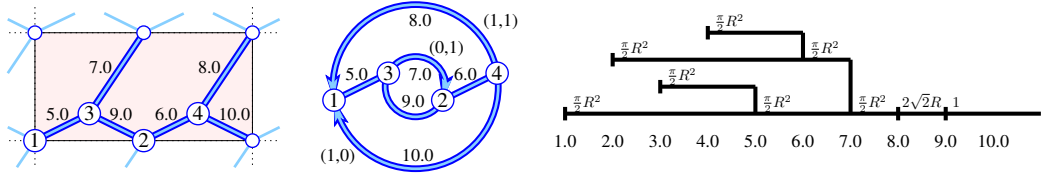}
  \caption{\footnotesize \emph{From left to right:} a perturbation of the periodic graph in the left panel of Figure~\ref{fig:TwoTreesTwo}, the corresponding quotient graph shown with filter values and shift vectors, and the resulting periodic merge tree.
  Its interleaving pseudo-distance to the periodic merge tree in the left panel of Figure~\ref{fig:TwoTreesTwo} is $1.0$.
  This can be seen by first splintering the tree in the left panel to obtain the tree in the right panel of Figure~\ref{fig:TwoTreesTwo}, and then comparing the latter to the periodic merge tree in this figure.
  The intermediate step is necessary because after the perturbation, the graph is no longer periodic with respect to the initially used square lattice.}
  \label{fig:PerturbedTree}
\end{figure}
\begin{theorem}[Stability of Periodic Merge Trees]
  \label{thm:stability_of_periodic_merge_trees}
  Let $\Lambda \subseteq \Rspace^d$ be a $d$-dimensional lattice, and $\filter, \filteraux \colon K \to \Rspace$ two $\Lambda$-periodic filters.
  Then $\Idist{\perMTree{\filter}{\Lambda}}{\perMTree{\filteraux}{\Lambda}} \leq \Maxdist{\filter}{\filteraux}$ and hence $\Jdist{[\perMTree{\filter}{\Lambda}]}{[\perMTree{\filteraux}{\Lambda}]} \leq \Maxdist{\filter}{\filteraux}$.
\end{theorem}
\begin{proof}
  Writing $\MTree = \perMTree{\filter}{\Lambda}$ and $\MTree' = \perMTree{\filteraux}{\Lambda}$ for the two periodic merge trees, we prove $\Idist{\MTree}{\MTree'} \leq \Maxdist{\filter}{\filteraux}$,
  and by definition of the interleaving pseudo-distance in terms of the interleaving distance, we have $\Jdist{[\MTree]}{[\MTree']} \leq \Idist{\MTree}{\MTree'}$, which then implies the claimed inequality.
  Write $\ee = \Maxdist{\filter}{\filteraux}$, and let $\AAA$ and $\AAA'$ be connected components of the sublevel sets of $\filter/\Lambda$ and $\filteraux/\Lambda$ at height $\Height(\AAA) = \Height'(\AAA') = t$, respectively.
  Since $\ee$ is the $L_\infty$-distance between $\filter$ and $\filteraux$, there are unique connected components $\BBB$ and $\BBB'$ of the sublevel sets of $\filter/\Lambda$ and $\filteraux/\Lambda$ at height $t+\ee$, respectively, such that $\AAA \subseteq \BBB'$ and $\AAA' \subseteq \BBB$.
  By the stability theorem for merge trees in \cite{MBW13}, the maps $\varphi \colon \MTree \to \MTree'$ and $\psi \colon \MTree' \to \MTree$ defined by $\varphi (\AAA) = \BBB'$ and $\psi (\AAA') = \BBB$ are $\ee$-compatible as far as the two merge trees are concerned; that is: they satisfy Conditions~{\sf (i)} and {\sf (ii)} of Definition~\ref{dfn:interleaving_distance}.

  \smallskip
  It remains to prove Condition~{\sf (iii)}.
  Since $\AAA \subseteq \BBB'$, an argument analogous to the proof of Lemma~\ref{lem:monotonicity} yields $\Frequency'(\BBB') \leq \Frequency(\AAA)$ and, by symmetry, $\Frequency(\BBB) \leq \Frequency'(\AAA')$.
  Therefore, $\varphi$ and $\psi$ are $\ee$-compatible with respect to the two periodic merge trees, so $\Idist{\MTree}{\MTree'} \leq \ee$ by definition of the interleaving distance between these trees.
\end{proof}

See Figure~\ref{fig:PerturbedTree} for an illustration of the stability in which the perturbation forces a coarser lattice and thus an intermediate step of splintering.
With the help of Theorem~\ref{thm:invariance_of_periodic_merge_trees}, the statement about the interleaving pseudo-distance in Theorem~\ref{thm:stability_of_periodic_merge_trees} can be extended to periodic filters $\filter \colon K \to \Rspace$ and $\filteraux \colon L \to \Rspace$, provided $K$ and $L$ have the same underlying space and a common periodic refinement.

\section{The Periodic Barcode}
\label{sec:5}

Motivated by simplifying the computation of distance, we strip information off the periodic merge tree to construct the periodic analog of the $0$-th barcode or persistence diagram.
Following the structure of the previous section, we introduce a new notion of distance right after defining the concept and prove invariance and stability thereafter.

\subsection{Definition by Construction}
\label{sec:5.1}

In a nut-shell, we get the periodic $0$-th barcode by calling the beams of the periodic merge tree bars while dropping the vertical segments that connect them.
Recall, however, that different epochs along a beam have different shadow monomials, which decrease from left to right, and we need to account for them by possibly substituting more than one bar for each beam.
Specifically, when the monomial decreases from $s \nu_{d-p} R^{d-p}$ to $t \nu_{d-q} R^{d-q}$ at some interior point of a beam, this point marks the end of an era, if $d-q < d-p$, and the end of an epoch inside an era, if $d-q = d-p$.
This change is brought about by $s \nu_{d-p} R^{d-p} - t \nu_{d-q} R^{d-q}$ deaths giving rise to polynomial decorations; see the step from the top left panel to the bottom left panel in Figure~\ref{fig:barcode}.
We recover monomials by introducing two bars that end at this point, one decorated with $s \nu_{d-p} R^{d-p}$ and the other with $-t \nu_{d-q} R^{d-q}$.
This enables the classification of the bars into $d+1$ eras, which is desirable since different growth-rates render the coefficients of the corresponding monomials incommensurable.
We thus obtain the $0$-th barcode as an ensemble of $d+1$ collections of bars, one for each era; see the middle panels in Figure~\ref{fig:barcode}. 
The alternative visualization as an ensemble of $d+1$ points is displayed in the right panels of the same figure.
Within each era we, drop the factor $\nu_{d-p} R^{d-p}$ of the decoration that all bars share.
\begin{figure}[hbt]
  \centering \vspace{0.0in}
  \resizebox{!}{2.1in}{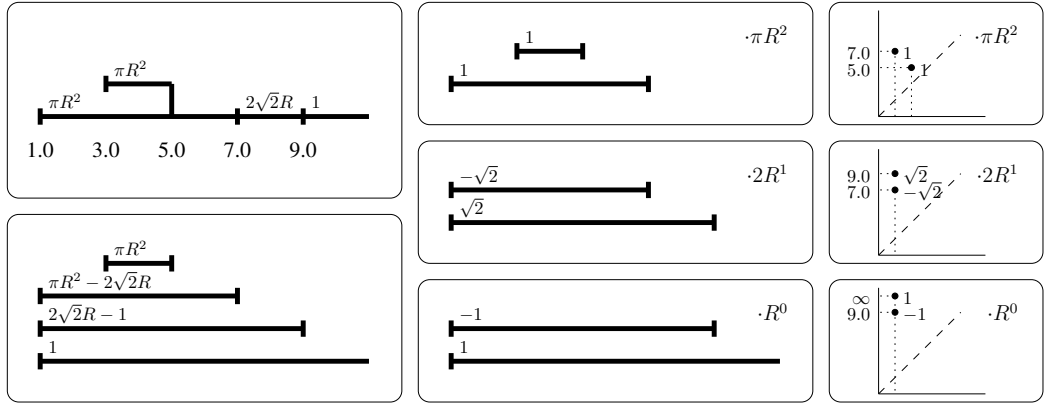}
  \caption{\footnotesize \emph{Left:} the periodic merge tree of the left panel in Figure~\ref{fig:TwoTreesTwo} and below its decomposition into labeled bars. \emph{Middle:} The periodic barcode split into groups of bars per era. \emph{Right:} the corresponding ensemble of periodic persistence diagrams.}
  \label{fig:barcode}
\end{figure}

\smallskip
With this in mind, we compute the periodic barcode from the periodic merge tree one epoch at a time.
In particular, consider a beam with left endpoint $\CCC \in \MTree$, and an epoch on this beam that starts at $\AAA \in \MTree$ and ends entering $\BBB \in \MTree$.
Letting $\Frequency (\AAA) = s \nu_{d-p} R^{d-p}$ be the monomial decorating this epoch, we introduce two bars into the \emph{$(d-p)$-ary $0$-th barcode}: $[\Height (\CCC), \Height (\BBB))$ with \emph{multiplicity} $s$ and $[\Height (\CCC), \Height (\AAA))$ with \emph{multiplicity} $-s$, but note that the second bar might be empty, in which case we skip it.
The first bar tells us that $s \nu_{d-p} R^{d-p}$ (minus the monomial of the next epoch) of the shadows born at $\Height (\CCC)$ die at $\Height (\BBB)$, and the second bar tells us that the monomial of the previous epoch minus $s \nu_{d-p} R^{d-p}$ of the shadows born at $\Height (\CCC)$ die at $\Height (\AAA)$.
This straightforward conversion takes time proportional to the number of epochs and gives a collection of bars with possibly negative multiplicities.
We refer to the collection of the $d+1$ different $(d-p)$-ary $0$-th barcodes as the \emph{periodic $0$-th barcode}, or the \emph{periodic $0$-th persistence diagram}, if we choose to visualize the information with points rather than bars.
Similar to the notation for the periodic merge tree, we write $\perBCode{\filter}{\Lambda}$ for the periodic $0$-th barcode of $\filter$ and $\Lambda$.

\smallskip
Setting $\barRspace = (-\infty, \infty]$, we note that for each era, the corresponding bars with multiplicities can be viewed as a function $\Rspace \times \barRspace \to \Rspace$, which maps a point $(b,d)$ to the (possibly negative) sum of the multiplicities of all bars $[b,d)$.
This sum is $0$ if there are no such bars.
While some vectorization methods for barcodes require the multipliticites to be non-negative integers, some other standard vectorization methods, such as persistence images \cite{Ada17}, do not need such requirements and can thus be used for periodic barcodes.

\subsection{Alternating Wasserstein Distance}
\label{sec:5.2}

While the interleaving distance between periodic merge trees is difficult to find, it is straightforward to compute a meaningful distance between periodic $0$-the barcodes, as we will shortly see.
We begin by formulating the Wasserstein distance between two periodic barcodes as a solution to an optimal transport problem.
The bars in classic barcodes come with positive integer multiplicities, for which the Wasserstein distance is a solution to an optimal assignment problem; see \cite[Section~VIII.4]{EdHa10} for an algorithm.
Contrast this to the periodic setting, in which we have real and possibly negative multiplicities, which requires a solution to the more general optimal transport problem; see \cite{PeCu19} for comprehensive background.
We will work with the linear version, commonly referred to as the \emph{$1$-Wasserstein distance}, which is a weighted sum of distances between points.
The main reason is that this version satisfies $\Wasser{1}{\xi}{\eta} = \Wasser{1}{\xi + \zeta}{\eta + \zeta}$, and a similar relation does not hold if we work with powers $q > 1$ of the distances.
This property allows for the removal of negative multiplicities by adding them on both sides, which is also the intuition behind the definition of the alternating $1$-Wasserstein distance in \eqref{eqn:Wasser2}.
Recall that the \emph{support} a function $f \colon \Rspace \times \barRspace \to \Rspace$, denoted $\support{f}$, are the points in the domain for which $f$ is non-zero.
\begin{definition}[Alternating Wasserstein Distance]
  \label{dfn:alternating_Wasserstein_distance}
  Let $\xi, \eta \colon \Rspace \times \barRspace \to \Rspace$ be two non-negative multiplicity functions, each with finite support.
  Then the \emph{$1$-Wasserstein distance} between $\xi$ and $\eta$ is
  \begin{align}
    \Wasser{1}{\xi}{\eta} &= \inf\nolimits_{\Tplan, \Xplan, \Yplan}
      \left[ \sum\nolimits_{x,y} \Tplan(x,y) \Onedist{x}{y} + \sum\nolimits_x \Xplan(x) \delta(x) + \sum\nolimits_y \Yplan(y) \delta (y) \right] ,
        \label{eqn:Wasser1}
  \end{align}
  in which $\delta(x) = |x_2-x_1|$ is the vertical distance of $x = (x_1,x_2)$ to the diagonal, and the infimum is taken over all functions $\Tplan \colon \support{\xi} \times \support{\eta} \to [0,\infty)$, $\Xplan \colon \support{\xi} \to [0, \infty)$, and $\Yplan \colon \support{\eta} \to [0, \infty)$ that satisfy
  \begin{align}
    \Xplan(x) + \sum\nolimits_{y \in \support{\eta}} \Tplan(x,y) &= \xi(x) \mbox{\rm ~~~and~~~}
    \Yplan(y) + \sum\nolimits_{x \in \support{\xi}} \Tplan(x,y) = \eta(y) ,
  \end{align}
  for all $x \in \support{\xi}$ and $y \in \support{\eta}$.
  Without assuming the non-negativity of the multiplicities, we write $\xi = \xi^+ - \xi^-$ and $\eta = \eta^+ - \eta^-$, in which $\xi^+, \xi^-, \eta^+, \eta^-$ are non-negative.
  The more general \emph{alternating $1$-Wasserstein distance} between $\xi$ and $\eta$ is
  \begin{align}
    \Wasserpm{1}{\xi}{\eta} &= \Wasser{1}{\xi^+ + \eta^-}{\xi^- + \eta^+} .
      \label{eqn:Wasser2}
  \end{align}
\end{definition}
Note that the relation $\Wasser{1}{\xi}{\eta} = \Wasser{1}{\xi + \zeta}{\eta + \zeta}$ implies that $\Wasserpm{1}{\xi}{\eta}$ does not depend on the choice of $\xi^+, \xi^-, \eta^+, \eta^-$.
We may therefore assume that $\xi^+$ and $\xi^-$ have disjoint supports, and similar for $\eta^+$ and $\eta^-$.
We will use this property in the proof of the triangle inequality for the alternating $1$-Wasserstein distance.

\smallskip
Write $\Delta$ for the diagonal in $\Rspace^2$; that is: the points $x = (x_1, x_2)$ with $x_1 = x_2$.
For non-negative multiplicity functions whose support avoids $\Delta$, the $1$-Wasserstein distance is an extended metric, meaning it can also take $\infty$ as a value, namely between persistence pairs with finite and infinite deaths.
We use this property to prove that the alternating $1$-Wasserstein distance is also an extended metric.
\begin{lemma}[$W_1^{\pm}$ is Extended Metric]
  \label{lem:Wpm1_is_extended_metric}
  Let $\xi, \eta, \zeta \colon \Rspace \times \barRspace \to \Rspace$ be multiplicity functions, each with finite support avoiding $\Delta$.
  Then
  \smallskip \begin{itemize}
    \item $\Wasserpm{1}{\xi}{\eta} \geq 0$ and $\Wasserpm{1}{\xi}{\eta} > 0$ iff $\xi \neq \eta$ (positivity);
    \item $\Wasserpm{1}{\xi}{\eta} = \Wasserpm{1}{\eta}{\xi}$ (symmetry);
    \item $\Wasserpm{1}{\xi}{\eta} + \Wasserpm{1}{\eta}{\zeta} \geq \Wasserpm{1}{\xi}{\zeta}$ (triangle inequality).
  \end{itemize}
\end{lemma}
\begin{proof}
  Write $\xi = \xi^+ - \xi^-$, $\eta = \eta^+ - \eta^-$, $\zeta = \zeta^+ - \zeta^-$, in which $\xi^+, \xi^-$ are non-negative with disjoint supports, and similar for $\eta^+, \eta^-$ and $\zeta^+, \zeta^-$.
  To show positivity, we note that the right-hand side of \eqref{eqn:Wasser2} can only vanish when $\xi^+ + \eta^- = \xi^- + \eta^+$.
  Since $\support{\xi^+} \cap \support{\xi^-} = \support{\eta^+} \cap \support{\eta^-} = \emptyset$, this implies $\xi^+ = \eta^+$ and $\xi^- = \eta^-$ and therefore also $\xi = \eta$.
  Symmetry is implied by the symmetry of Definition~\ref{dfn:alternating_Wasserstein_distance}.
  Finally, we use the triangle inequality for positive multiplicities to prove the same for possibly negative multiplicities:
  \begin{align}
    \Wasserpm{1}{\xi}{\eta} + \Wasserpm{1}{\eta}{\zeta}
       &= \Wasser{1}{\xi^+ + \eta^-}{\xi^- + \eta^+}
       + \Wasser{1}{\eta^+ + \zeta^-}{\eta^- + \zeta^+} \\
       &= \Wasser{1}{\xi^+ + \eta^- + \zeta^-}{\xi^- + \eta^+ + \zeta^-} \\
       &~~~~+ \Wasser{1}{\xi^- + \eta^+ + \zeta^-}{\xi^- + \eta^- + \zeta^+} \\
       &\geq \Wasser{1}{\xi^+ + \eta^- + \zeta^-}{\xi^- + \eta^- + \zeta^+} \\
       &= \Wasser{1}{\xi^+ + \zeta^-}{\xi^- + \zeta^+} ,
  \end{align}
  in which the last term is equal to $\Wasserpm{1}{\xi}{\zeta}$.
  In summary, the alternating Wasserstein distance satisfies positivity, symmetry, and the triangle inequality and is therefore an extended metric.
\end{proof}

We define the alternating Wasserstein distance of two periodic $0$-th barcodes as the sum over the alternating Wasserstein distances of the $(d-p)$-ary barcodes.

\subsection{Invariance of Periodic Barcodes}
\label{sec:5.3}

In this subsection, we prove that periodic barcodes do not depend on the lattice used in their construction.
This is in contrast to periodic merge trees, for which the splintering relation is needed to get invariant equivalence classes of such trees.
We begin by showing that splintering does not affect the periodic barcode derived from a periodic merge tree.
\begin{lemma}[Splintering Preserves Periodic Barcodes]
  \label{lem:splintering_preserves_periodic_barcodes}
  Let $\BCode$ and $\BCode'$ be the periodic $0$-th barcodes derived from the periodic merge trees $\MTree$ and $\MTree'$.
  If $\MTree'$ splinters $\MTree$, then $\BCode' = \BCode$.
\end{lemma}
\begin{proof}
  To get $\BCode$, we collect two bars from each epoch of $\MTree$, one with positive and the other with negative multiplicity, but note that the second bar may be empty.
  Let $\AAA, \BBB \in \MTree$ delimit an epoch decorated with $s \nu_{d-p} R^{d-p}$, and let $\CCC \in \MTree$ be the leftmost point on the same beam.
  By the elder rule applied in the construction of the merge tree, $\Height (\CCC)$ is the minimum height of any point in the subtree rooted at $\AAA$.
  The two bars are $[\Height (\CCC), \Height (\BBB))$ with multiplicity $s$ and $[\Height (\CCC), \Height (\AAA))$ with multiplicity $- s$.
  If we subdivide the epoch into intervals we call \emph{ages} and construct the bars from each age, we get the same result because all bars ending at points strictly between $\AAA$ and $\BBB$ cancel in pairs.
  This motivates us to refine the epochs of $\MTree$ and $\MTree'$ into ages such that
  \smallskip \begin{enumerate}[(i)]
    \item none of the ages contains the value of critical event in its interior;
    \item the preimages of an age in $\MTree$ are ages in $\MTree'$.
  \end{enumerate} \smallskip
  Consider an age with shadow monomial $s \nu_{d-p} R^{d-p}$ in $\MTree$, let $N$ be the number of its preimages in $\MTree'$, and observe that each of these preimages is an age with shadow monomial $\frac{s}{N} \nu_{d-p} R^{d-p}$.
  The bars generated by the age in $\MTree$ and its preimages in $\MTree$ are the same, except that the former have multiplicities $\pm s$, while the latter have multiplicities $\pm \frac{s}{N}$.
  Adding the contributions of the $N$ preimages amounts to the same two bars with the same multiplicities.
  Since this is true for all ages, we conclude that $\BCode = \BCode'$.
\end{proof}

The invariance of periodic barcodes is an immediate consequence of Lemma~\ref{lem:splintering_preserves_periodic_barcodes} and Lemma~\ref{lem:splintering_from_sublattice}.
Indeed, if $\Lambda', \Lambda'' \subseteq \Lambda$ are two $d$-dimensional sublattices, and 
$\perMTree{\filter}{\Lambda'}$, $\perMTree{\filter}{\Lambda''}$ 
are the corresponding periodic merge trees, then they both splinter 
$\perMTree{\filter}{\Lambda}$
by Lemma~\ref{lem:splintering_from_sublattice}.
By Lemma~\ref{lem:splintering_preserves_periodic_barcodes}, the periodic barcodes derived from these periodic merge trees satisfy $\perBCode{\filter}{\Lambda'} = \perBCode{\filter}{\Lambda} = \perBCode{\filter}{\Lambda''}$.
We formulate for later reference:
\begin{corollary}[Invariance of Periodic Barcodes]
  \label{cor:invariance_of_periodic_barcodes}
  Let $\Lambda \subseteq \Rspace^d$ be a $d$-dimensional lattice, $\filter \colon K \to \Rspace$ a $\Lambda$-periodic filter, and $\Lambda', \Lambda'' \subseteq \Lambda$ two $d$-dimensional sublattices.
  Then $\perBCode{\filter}{\Lambda'} = \perBCode{\filter}{\Lambda''}$.
\end{corollary}

\subsection{Size of Multiplicities}
\label{sec:5.4}

To prepare the proof of stability for periodic barcodes, this section studies how large the multiplicities can get.
We thus consider the maximum multipicity of a bar in the $0$-barcode of $\filter$ and $\Lambda$,
\begin{align}
  \multiplicity (\filter, \Lambda)
  &= \max \left\{ \frac{\volume{p}{\Lambda_\CCC}}{\volume{d}{\Lambda}} \mid \CCC \in \perMTree{\filter}{\Lambda} \right\} ,
\end{align}
and prove an upper bound that depends on the dimension, $d$, the number of edges in the quotient complex, $m$, the magnitude of the shift vectors, $D$, but also on the basis of $\Lambda$.
As before, we write $U$ for the matrix whose columns are the vectors in the basis of $\Lambda$.
Recall that $U^{-1}$ maps $\Lambda$ to $\Zspace^d$, and that the \emph{operator norm}, $\normop{U^{-1}}{} = \max\nolimits_{x \in \Rspace^d} {\normtwo{U^{-1}(x)}}/{\normtwo{x}}$, is the largest singular value of $U^{-1}$ and quantifies the maximum stretching experienced by any vector $x \in \Rspace^d$.
\begin{lemma}[Maximum Multiplicity]
\label{lem:maximum_multiplicity}
  Let $\Lambda \subseteq \Rspace^d$ be a $d$-dimensional lattice, $U$ the matrix of basis vectors, $\filter \colon K \to \Rspace$ a $\Lambda$-periodic filter on a $\Lambda$-periodic complex, $m$ the number of edges, and $D$ the maximum magnitude of their shift vectors in $K / \Lambda$.
  Then $\multiplicity (\filter, \Lambda) \leq (d^{2.5}Dm \normop{U^{-1}}{} )^d$.
\end{lemma}
\begin{proof}
  Write 
  $\MTree = \perMTree{\filter}{\Lambda}$,
  let $\CCC \in \MTree$ be arbitrary, set $p = \dime{\Lambda_\Gamma}$, and let $w_1, w_2, \ldots, w_p$ be a basis of $\Lambda_\CCC$.
  Write $v_i = U^{-1} (w_i)$.
  There are different ways to choose the basis.
  For example, the columns of the reduced matrix in Hermite normal form would be a valid choice.
  However, to get a better bound, we choose $v_i$ as the output of a different reduction algorithm, described in \cite{KlRe23}, which we apply to the matrix of size $d \times m$ whose columns are the drift vectors of the simple cycles that correspond to the fewer than $m$ catenation events in the subtree rooted at $\CCC$.
  As argued in the proof of Lemma~\ref{lem:magnitude_of_basis},
  the magnitude of this matrix is at most $Dm$, and according to \cite{KlRe23}, the magnitude of the reduced matrix---which is not necessarily in Hermite normal form---is at most $d^2 Dm$.
  It follows that the Euclidean norm of each column vector satisfies $\normtwo{v_i} \leq \sqrt{d} (d^2 Dm) = d^{2.5}Dm$.
  In order to write $\volume{p}{\Lambda_\CCC}$ as a $d$-dimensional volume, let $w_{p+1}, w_{p+2}, \ldots, w_{d}$ be unit vectors orthogonal to each other and to $\Lambda_\CCC$, so that $\volume{p}{\Lambda_\CCC} = \volume{d}{\Lambda_\CCC'}$, with $\Lambda_\CCC'$ spanned by $w_1, w_2, \ldots, w_d$.
  Since a linear transformation changes the $d$-dimensional volume of any subset of $\Rspace^d$ in the same way, we get
  \begin{align}
    \frac{\volume{d}{\Lambda_\CCC'}}{\volume{d}{\Lambda}} 
      &= \frac{\volume{d}{U^{-1}(\Lambda_\CCC'})}{\volume{d}{U^{-1}(\Lambda})} 
       = \volume{d}{U^{-1}(\Lambda_\CCC'})
      \leq \prod_{i=1}^d\normtwo{v_i} 
        \label{eqn:multiplicity1} \\
      &\leq \left( d^{2.5}Dm \right)^p \cdot \normop{U^{-1}}{d-p} \leq \left( d^{2.5}Dm \normop{U^{-1}}{}\right)^d ,
        \label{eqn:multiplicity2}
  \end{align}
  in which we get the third term in \eqref{eqn:multiplicity1} by recalling that $U^{-1} (\Lambda) = \Zspace^d$, and we get the first term in \eqref{eqn:multiplicity2} by using Lemma~\ref{lem:magnitude_of_basis} for the first $p$ and the operator norm for the last $d-p$ basis vectors.
\end{proof}
\noindent \emph{Remark about dependence on basis.}
Both $D$ and $\normop{U^{-1}}{}$ depend on the basis we choose for $\Lambda$, but by Corollary~\ref{cor:invariance_of_periodic_barcodes}, the barcode and therefore the multiplicities of its bars do not.
To gain a better understanding of the size of the multiplicities, it would therefore we useful to know how small $D$ and $\normop{U^{-1}}{}$ can be made by careful choice of a basis.
Indeed, the bound in Lemma~\ref{lem:maximum_multiplicity} seems excessively pessimistic.
It would be interesting to find out what the maximum multiplicity is in practice, e.g.\ on average for the periodic structures collected in popular crystallographic databases.

\subsection{Stability of Periodic Barcodes}
\label{sec:5.5}

The \emph{cellular $\ell_1$-distance} between filters $\filter/\Lambda, \filteraux/\Lambda \colon K / \Lambda \to \Rspace$ is the sum of absolute differences over all simplices, $\Onedist{ \filter/\Lambda }{ \filteraux/\Lambda } = \sum_{\sigma \in K/\Lambda} \abs{\filter/\Lambda(\sigma) - \filteraux/\Lambda (\sigma)}$, but note that this is different from the $L_1$-distance between the two functions, which is be the integral of their point-wise differences.
By comparing the $1$-Wasserstein distance of two non-periodic barcodes with the cellular $\ell_1$-distance between the corresponding filters, Skraba and Turner~\cite{SkTu20} prove stability with Lipschitz constant $1$.
Following this approach, we prove stability with respect to the cellular $\ell_1$-distance in the periodic setting, but with worse Lipschitz constant.
Observe first that the periodic $0$-th barcode is invariant under changing how ties in the ordering of the simplices are broken.
Indeed, the periodic merge tree is invariant, and the asymmetry introduced by the elder rule has no effect on the periodic barcode computed from the periodic merge tree.
This invariance is of course necessary for the stability of periodic barcodes, which we prove next.
\begin{theorem}[Stability of Periodic Barcodes]
  \label{thm:stability_of_periodic_barcodes}
  Let $\Lambda \subseteq \Rspace^d$ be a $d$-dimensional lattice, $U$ the matrix of a basis of $\Lambda$, and $\filter, \filteraux \colon K \to \Rspace$ two $\Lambda$-periodic filters.
  Then
  \begin{align}
    W_1^{\pm} (\perBCode{\filter}{\Lambda}, \perBCode{\filteraux}{\Lambda})
       &\leq 2 (d+1) \multiplicity_0 \cdot \Onedist{ \filter/ \Lambda }{ \filteraux / \Lambda } ,
  \end{align}
  in which $\multiplicity_0 = (d^{2.5}Dm \normop{U^{-1}}{})^d$ is the upper bound on the maximum multiplicity from Lemma~\ref{lem:maximum_multiplicity}, with $m$ the number of edges of $K / \Lambda$, and $D$ the maximum magnitude of any of their shift vectors.
\end{theorem}
\begin{proof}
  The proof is analogous to that of \cite[Lemma 4.7 and Theorem 4.8]{SkTu20}, which itself follows the strategy of the proof of \cite[Combinatorial Stability Theorem]{CEM06}. 
  It consists of two steps: the first assumes that $\filter$ and $\filteraux$ are compatible to a common ordering of the simplices, and the second analyzes the transpositions necessary to make them compatible to such an ordering, if they are not.

  \smallskip \noindent \emph{Step 1.}
  Assume the filters $\filter$ and $\filteraux$ allow for a total order of the simplices that is compatible to both; that is: $\sigma_i \subseteq \sigma_j$, $\filter(\sigma_i) < \filter(\sigma_j)$, and $\filteraux(\sigma_i) < \filteraux(\sigma_j)$ all imply $i < j$.
  Let ${\rm index} \colon K/\Lambda \to \Rspace$ be the filter that assigns to each simplex its index in this order, and construct 
  $\MTreeIndex = \perMTree{{\rm index}}{\Lambda}$.
  Replacing the height $\Height (\CCC) = i$ of a critical event by $\Height (\CCC) = \filter (\sigma_i)$ and $\Height (\CCC) = \filteraux (\sigma_i)$, respectively, we obtain two periodic merge trees, $\MTree_\filter$ and $\MTree_\filteraux$, whose only difference is the height function.
  However, since $\filter (\sigma_i) = \filter (\sigma_j)$ is possible even if $i \neq j$, some epochs in $\MTreeIndex$ may correspond to \emph{empty epochs} in $\MTree_\filter$, and similarly for $\MTree_\filteraux$.
  When we construct the periodic barcodes, $\BCode_\filter = \perBCode{\filter}{\Lambda}$ and $\BCode_\filteraux = \perBCode{\filteraux}{\Lambda}$, we use the empty epochs as well, while noting that each contributes two bars that cancel each other.
  The bijection between the epochs of $\MTree_\filter$ and $\MTree_\filteraux$ (both empty and non-empty) thus gives a bijection between $\BCode_\filter$ and $\BCode_\filteraux$, which is a transportation plan, $\Tplan \colon \support{\xi} \times \support{\eta} \to \Rspace$, as in Definition~\ref{dfn:alternating_Wasserstein_distance}.
    \footnote{Whenever the transportation plan maps to or from a specific point on the diagonal, we could instead use $\Xplan$ or $\Yplan$ to or from the closest point on the diagonal.
    For the sake of a simpler argument, we refrain from using $\Xplan$ and $\Yplan$ and rely solely on $\Tplan$, which leads to a possible over-estimation of the $1$-Wasserstein distance.}
    \footnote{Note also that instead of canceling the two bars contributed by an empty epoch, we keep both.
    This corresponds to looking at $\Wasser{1}{\xi+\zeta}{\eta + \zeta}$ instead of $\Wasser{1}{\xi}{\eta}$, but recall that they agree.}
  We prove that the cost of this particular transportation plan is bounded from above by $2 (d+1) \multiplicity_0 \cdot \Onedist{\filter/\Lambda}{\filteraux/\Lambda}$, which will imply
  \begin{align}
    W_1^{\pm} (\BCode_\filter, \BCode_\filteraux)
    &\leq  2 (d+1) \multiplicity_0 \cdot \Onedist{\filter/\Lambda}{\filteraux/\Lambda} .
      \label{eqn:transportationcost}
  \end{align}
  By construction, we compare two bars if they are born and die at the hands of the same two simplices.
  Their contribution to $\Tplan$ is therefore the difference in birth values plus the difference in death values, times the absolute multiplicity.
  Depending on the type of critical event a simplex causes, its value may mark more or fewer births or deaths.
  If $\sigma$ causes a catenation or merge event, its value marks the death of at most three bars, with the absolute multiplicity of each bounded by $\multiplicity_0$.
  So the total cost for transporting the corresponding death values from $\filter$ to $\filteraux$ is bounded by $3 \multiplicity_0 |\filter (\sigma) - \filteraux (\sigma)|$.
  If $\sigma$ causes an appearance, its value may mark the births of any number of bars.
  However, within each era, absolute multiplicities form a telescoping series whose sum is bounded by $2 \multiplicity_0$.
  Since there are $d+1$ eras, the cost for transporting the corresponding birth values from $\filter$ to $\filteraux$ is bounded by $2 (d+1) \multiplicity_0 \cdot |\filter (\sigma) - \filteraux (\sigma)|$.
  Adding the cost for the death and the birth values implies \eqref{eqn:transportationcost}.

  \smallskip \noindent \emph{Step 2.}
  If $\filter$ and $\filteraux$ are not necessarily compatible with the same total order, we look at the straight-line homotopy between the two filters: $\filter_t = (1-t) \filter + t \filteraux$ for $t \in [0,1]$ and observe that
  \smallskip \begin{itemize}
    \item $\filter_0=\filter$ and $\filter_1 = \filteraux$;
    \item each $\filter_t$ is a periodic filter with respect to $\Lambda$;
    \item for each $0 \leq s \leq t \leq 1$, $\Onedist{\filter_{s}/ \Lambda}{\filter_{t}/ \Lambda} = (t-s) \Onedist{ \filter/ \Lambda }{ \filteraux / \Lambda }$;
    \item there are only finitely many values $0=t_0 < t_1 < \ldots < t_k = 1$ at which the total order compatible with $F_t$ changes.
  \end{itemize} \smallskip
  With this, we finish the proof by taking the sum of the bounds in \eqref{eqn:transportationcost} for each interval between two such values:
  \begin{align}
    W_1^{\pm} (\BCode_\filter, \BCode_\filteraux)
      &\leq \sum\nolimits_1^k W_1^{\pm} (\perBCode{\filter_{t_{i-1}}}{\Lambda}, \perBCode{\filter_{t_{i}}}{\Lambda}) \\
      &\leq 2 (d+1) \multiplicity_0 \sum\nolimits_1^k \Onedist{\filter_{t_{i-1}}/ \Lambda}{\filter_{t_{i}}/ \Lambda} \\
      &= 2 (d+1) \multiplicity_0  \cdot \Onedist{ \filter/ \Lambda }{ \filteraux / \Lambda } ,
  \end{align}
  which is the claimed inequality.
\end{proof}

\noindent \emph{Remark about invariance.}
When a perturbation breaks a symmetry and thus forces a coarser lattice, like in Figure~\ref{fig:PerturbedTree}, Theorem~\ref{thm:stability_of_periodic_barcodes} can be applied to the filter before and after the perturbation, both considered under the coarse lattice.
By Corollary~\ref{cor:invariance_of_periodic_barcodes}, passing to an unnecessarily coarse lattice does not change the periodic barcode.

\smallskip \noindent \emph{Remark about distortion.}
For a $\Lambda$-periodic $\filter \colon K \to \Rspace$, and an arbitrarily small $\ee > 0$, a uniform scaling of $K$ and $\Lambda$ by $1+\varepsilon$ that preserves the filter values yields a periodic merge tree with slightly smaller shadow monomials.
Hence, there is no arbitrarily small interleaving with the original periodic merge tree.
However, the $1$-Wasserstein distance between the corresponding periodic barcodes is less sensitive to small differences in the multiplicities and therefore the better choice in coping with the effect of slight scalings, or more general affine linear transformations with singular values close to $1$.

\smallskip \noindent \emph{Remark about Delaunay triangulations.}
Note that the perturbation of the vertices of a $2$-dimensional Delaunay triangulation may change the filter values but possibly also the underlying graph, which happens only when four points lie on a common circle so one diagonal of a convex quadrangle \emph{flips} to the other diagonal.
In this case, neither diagonal contributes to the merge trees, because their endpoints are connected via shorter edges of the quadrangle.
Hence, a diagonal flip does not affect the periodic merge tree. 
Since Theorems~\ref{thm:stability_of_periodic_merge_trees} and \ref{thm:stability_of_periodic_barcodes} apply to each part of the straight line homotopy between Delaunay triangulations during which the graph structure does not change---and there is no contribution due to the diagonal flips---the two theorems can be used to bound the distance between our descriptors of two Delaunay triangulations with possibly non-isomorphic edge-skeletons.

\section{A Working Example}
\label{sec:6}

This section illustrates the concepts introduced in Sections~\ref{sec:3} to \ref{sec:5} using the periodic graph in Figure~\ref{fig:3D-graph} as a running example.
Its quotient in the $3$-dimensional torus has five vertices, labeled from $1$ to $5$, and eight edges, each labeled with its filter value; see Figure~\ref{fig:3D-quotient}.
If an edge has shadows whose endpoints lie in different copies of the unit cell, we choose an arbitrary direction and label the directed edge with the shift vector that locates the unit cell of the target vertex relative to that of the source vertex.
In Figure~\ref{fig:3D-quotient}, two of the blue edges connect to the same vertex twice, and they correspond to two families of parallel lines in $\Rspace^3$: with diagonal direction defined by the drift vector $(1,1,0)$, and with vertical direction defined by the drift vector $(0,0,1)$.
\begin{figure}[hbt]
  \centering \vspace{-0.0in}
  \resizebox{!}{1.8in}{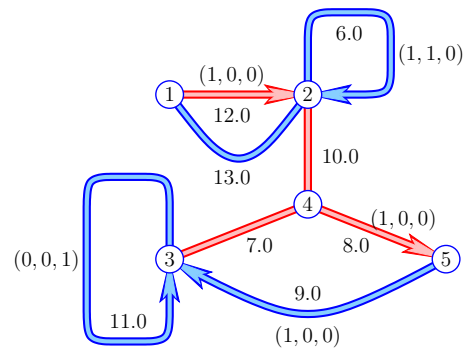}
  \vspace{-0.0in}
  \caption{\footnotesize The quotient of the graph in Figure~\ref{fig:3D-graph}.
  Each edge is labeled by its value in the filter, and some also by their shift vectors.
  The spanning tree edges, which merge components, are drawn straight and \emph{red}, and the others are drawn curved and \emph{blue}.
  Each edge whose shadow has endpoints in different copies of the unit cell has a non-zero shift vector associated to a directed version of the edge.}
  \label{fig:3D-quotient}
\end{figure}
The loop formed by the two edges connecting vertices $1$ and $2$ corresponds to another family of lines with direction defined by $(1,0,0)$.
More interesting is the loop formed by the three edges connecting vertices $3$, $4$, and $5$.
The drift vector is $(2,0,0)$, which implies that two copies of the strand (drawn curved in Figure~\ref{fig:3D-graph} for better visibility) form a double helix that passes through a linear sequence of unit cells.

\smallskip
Next we go through the motion of constructing the periodic merge tree incrementally, by adding one vertex or edge at a time in the order of increasing filter value.
The top panel of Figure~\ref{fig:3D-barcode} shows the final tree.
We pay special attention to the periodicity lattices, which are maintained incrementally and used to compute the shadow monomials decorating the beams of the tree.
For simplicity, we assume that the underlying lattice is $\Lambda = \Zspace^3$, whose unit cell is the unit cube with unit volume.
\begin{figure}[t]
  \centering \vspace{-0.1in}
  \resizebox{!}{3.3in}{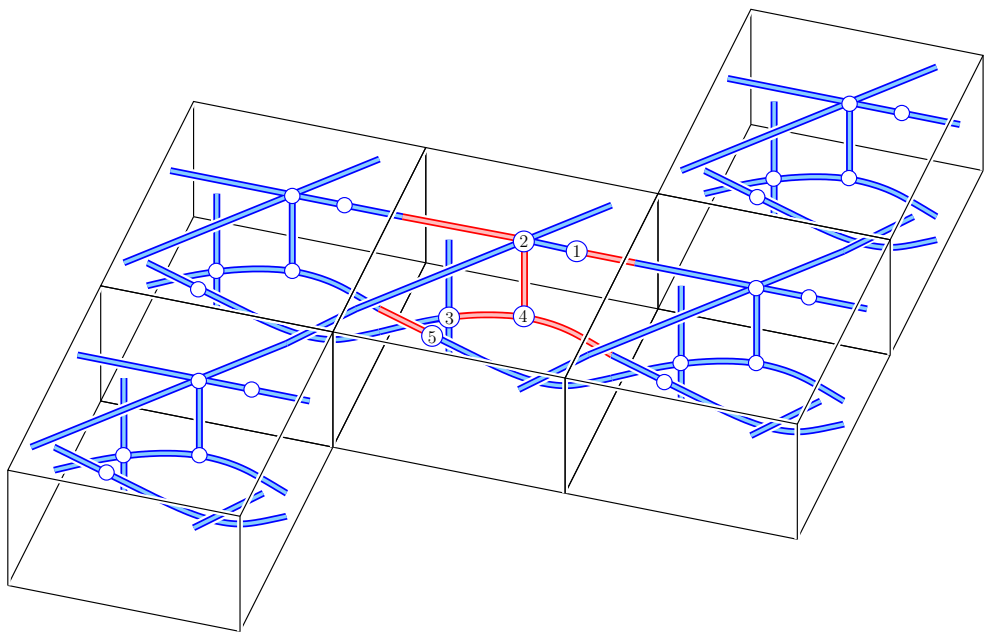}
  \vspace{-0.1in}
  \caption{\footnotesize A periodic graph in $\Rspace^3$ serving as a running example to illustrate the periodic merge tree and its construction.
  The portion of the graph inside five copies of the unit cell is shown.
  Inside the middle unit cell, the vertices are shown with labels and the edges in the spanning tree are drawn \emph{red}.
  The motif within this cell consists of a cross above a double helix, with a single edge connecting the two, and, in addition, a vertical line sharing a point with the helix. 
  This motif repeats periodically.
  The two lines of the cross meet at vertex $2$, and together with their periodic translates form a series of $2$-dimensional grids.
  The two strands of the double helix are periodic translates of each other and pass from left to right through copies of vertices $3$, $4$, $5$, in this cyclic sequence.
  For convenience, the vertex labels are also used as filter values, and the values of the edges are shown in Figure~\ref{fig:3D-quotient}, which displays the corresponding quotient, i.e.\ the graph in the $3$-dimensional torus.}
  \label{fig:3D-graph}
\end{figure}
\medskip \begin{description}
  \item[{\sc Steps 1 to 5:}] adding the five vertices, we get five connected components, each with trivial periodicity lattice, $\Lambda_1 = \Lambda_2 = \Lambda_3 = \Lambda_4 = \Lambda_5 = \{0\}$, in which we use the labels of the vertices as indices.
  Each of these lattices has $0$-dimensional volume $1$, which implies that the shadow monomial is $\frac{4 \pi}{3} R^3$; see the left ends of the beams in the top panel of Figure~\ref{fig:3D-barcode}.
  \item[{\sc Step 6:}] the edge with filter value $6.0$ connects vertex $2$ to itself, and its addition changes the periodicity lattice of the corresponding component to $\Lambda_2 = \Lambda((1,1,0))$.
  Its unit cell has $1$-dimensional volume $\sqrt{2}$.
  Note that the dimension of the periodicity lattice increases from $0$ to $1$, and the shadow monomial decreases from cubic to quadratic.
  The drop of the exponent means the component leaves the cubic and enters the quadratic era.
  \item[{\sc Steps 7 to 9:}] the edges with filter values $7.0, 8.0, 9.0$ form the helix in the periodic graph.
  The addition of the first edge merges the components of vertices $3$ and $4$, and the addition of the second edge merges this component with that of vertex $5$.
  These mergers do not affect the periodicity lattices.
  However, the third edge completes the helix with periodicity lattice $\Lambda_3 = \Lambda( (2,0,0) )$ and shadow monomial $2 \pi R^2$.
  The factor $2$ indicates that shifting a strand by $(1,0,0)$ produces a second strand, so we are dealing with a double and not a single helix.
  \item[{\sc Step 10:}] the edge with filter value $10.0$ connects the line from Step~6 with the helix from Step~9.
  The corresponding periodicity lattice is spanned by the vectors $(1,1,0)$ and $(2,0,0)$, which after reduction becomes $\Lambda_2 = \Lambda((1,1,0),(0,2,0))$.
  The $2$-dimensional volume of its unit cell is $2$, and so is the $1$-dimensional volume of the unit ball in $\Rspace^1$, which is why the shadow monomial is $4R$; see Figure~\ref{fig:3D-barcode}.
  Note that the component leaves the quadratic and enters the linear era.
  \item[{\sc Step 11:}] the edge with filter value $11.0$ connects vertex $3$ to itself.
  Its addition combines the $2$-dimensional structures from Step~10 to form a bigger, $3$-dimensional structure.
  Indeed, the periodicity lattice is spanned by vectors $(1,1,0)$, $(0,2,0)$, and $(0,0,1)$, whose unit cell has $3$-dimensional volume $2$.
  The fact that the volume is $2$ and not $1$ suggests that the two strands of the double helix are still in separate components.
  The corresponding shadow monomial is $2$, so the  component leaves the linear and enters the constant era.
  \item[{\sc Steps 12 and 13:}] until now, vertex $1$ was isolated, with its own beam in the periodic merge tree.
  The first edge with endpoints $1$ and $2$ merges the two components, keeping the larger of the two periodicity lattices unchanged, but assigning it now to beam $1$, by the elder rule. 
  The second such edge forms a loop with drift vector $(1,0,0)$.
  We use the reduction algorithm to select a basis from the three vectors spanning the lattice constructed in Step~11 and the drift vector of the new loop.
  The result is $\Lambda_1 = \Lambda((1,0,0),(0,1,0),(0,0,1))$, and the shadow monomial is $1$, so this step moves the one remaining component from a first to the second and final epoch in the constant era; see the right end of the tree in the top panel of Figure~\ref{fig:3D-barcode}.
\end{description}
\begin{figure}[hbt]
  \centering \vspace{0.1in}
  \resizebox{!}{4.2in}{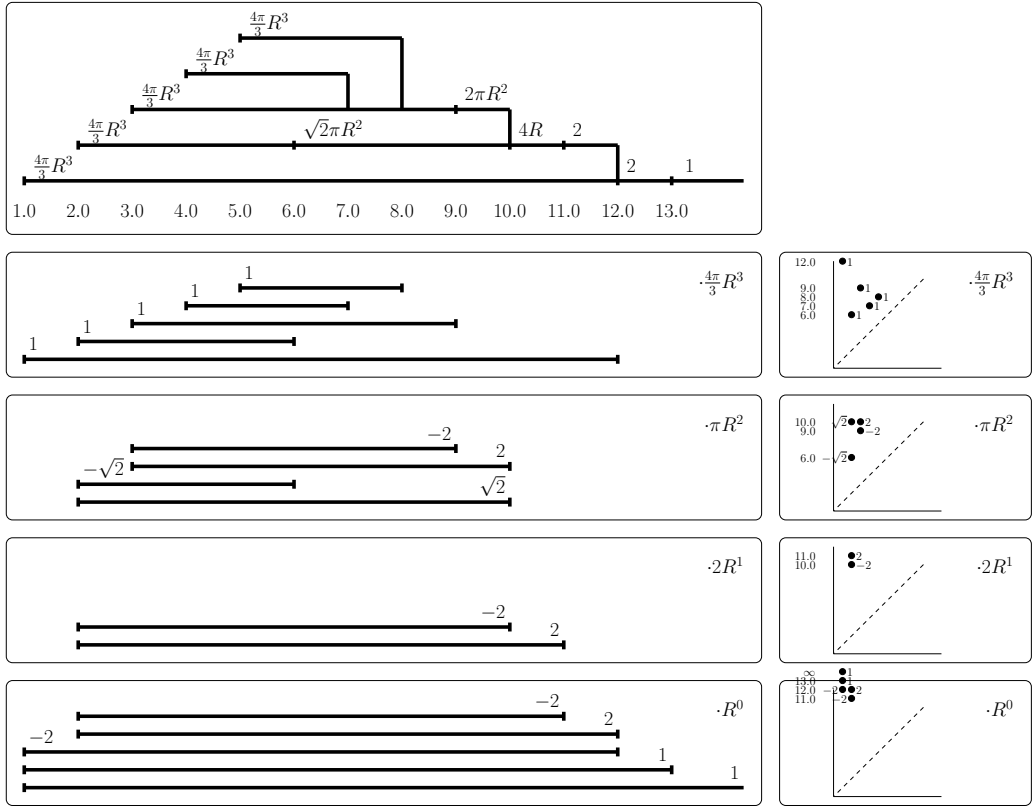}
  \vspace{-0.0in}
  \caption{\footnotesize \emph{Top:} the periodic merge tree of the filtered graph in Figure~\ref{fig:3D-graph}.
  The left endpoints of the beams correspond to vertices of the quotient graph, and the vertical segments correspond to edges in its spanning tree.
  The other edges form loops and thus do not affect the set of components but possibly alter the shadow monomials, which decorate the beams.
  \emph{Below the top on the left:} the four collections of bars labeled with their possibly negative multiplicities.
  \emph{On the right:} the corresponding four persistence diagrams with similarly labeled points.
  The bottommost constant diagram has a point with infinite vertical coordinate, which is drawn above its window.}
  \label{fig:3D-barcode}
\end{figure}

\noindent \emph{Remark about stability.}
If we alter the value of the second vertex from $2$ to $1+\ee$, for an arbitrarily small $\ee > 0$, the periodic merge tree and periodic barcode do not change combinatorially, but if we alter it to $1-\ee$, then this vertex swaps with the current first vertex and thus changes the combinatorics of the tree and barcode.
While the adjustment from $1+\ee$ to $1-\ee$ mostly affects birth values in the arbitrarily small interval between these value, there is one change worth noting:
the two epochs with shadow monomial $2$---which belong to different beams in the tree shown in Figure~\ref{fig:3D-barcode}---merge into one epoch on one beam.
Therefore, the points $(2,12)$ with multiplicity $2$ and $(1,12)$ with multiplicity $-2$ disappear.
This illustrates the importance of negative multiplicities, which are essential for stability as they facilitate the matching of such pairs using the alternating $1$-Wasserstein distance.

\section{Discussion}
\label{sec:7}

The main contribution of this paper is the extension of the persistent homology framework to the periodic setting, albeit only for connected components; that is: for homological dimension $0$.
The crucial ingredient in this extension is the notion of a \emph{shadow monomial}, which quantifies the growth-rate and density of the translates of a component.
The two new data structures are the \emph{periodic merge tree} and the \emph{periodic $0$-barcode}, with the former containing strictly more information and the latter offering computational advantages and an easier connection to machine learning tools;
see the text by Peyr\'{e} and Cuturi~\cite{PeCu19} for optimal transport algorithms and their application to data science.
Both data structures are invariant under isometries and a change of basis, which can be seen directly from the definition of the shadow monomial. 
While the periodic merge tree splinters when we pass to a sublattice of the original lattice, the periodic $0$-th barcode is invariant under such changes, and so are the equivalence classes of periodic merge trees. 
Both data structures are stable under perturbations.

\smallskip
The reported work suggests several avenues of future study aimed at broadening and deepening the capabilities of the method and thus make it more compelling in applications to the sciences.
\smallskip \begin{itemize}
  \item For the sake of convenience, we restricted ourselves to $d$-periodic filtrations in $\Rspace^d$. 
  However, the theory extends to the setting of $k$-periodic filtrations in $\Rspace^d$, by replacing all $d$ by $k$ in the definition of the shadow monomial.
  The non-periodic case, $k=0$, then yields shadow monomials equal to $1$ and thus the classical merge tree and $0$-th persistent homology.
  To what extent can the theory, and in particular Definition~\ref{dfn:shadow_monomial} and Lemma~\ref{lem:counting_inside_sphere}, be further extended to groups acting on topological spaces different from $\Rspace^d$?
  \item Further extend the results of this paper to homological dimensions beyond $0$.
  Some such capability can already be achieved by applying the periodic merge tree to the dual filtration, describing the complements of the original filtration.  
  While this idea is not new \cite{DeEd95}, a key difference is that the infinite periodic space
  is not compact, and therefore Alexander duality does not hold in the usual way.
  As a consequence, we not only get information about $(d-1)$- but also about $(d-2)$-dimensional homology.
  In $\Rspace^3$, $2$-dimensional homology (voids) of the original filtration corresponds to the cubic eras of the dual filtration.
  In contrast, the quadratic eras, 
  for example, describe infinitely long tunnels, so a particular kind of $1$-dimensional homology.
  \item Assuming success in extending periodic barcodes to higher homological dimension, it will be interesting to combine the periodic with the chromatic setting recently developed in \cite{CDES24}.
  Indeed, crystalline materials are periodic and typically consist of more than one type of atoms, which suggests both concepts in its analysis. 
\end{itemize} \smallskip
We conclude this paper with a concrete geometric question about the Delaunay triangulation of a locally finite $\Lambda$-periodic set in $\Rspace^d$:
does there exist a constant $D = D(d)$ and a basis of $\Lambda$ such that the shift vector of every edge in the Delaunay triangulation has magnitude at most $D$?
We recall that this constant is $D = 1$ if $\Lambda$ has a basis with pairwise orthogonal vectors, as proved in Section~\ref{sec:3.5}.
However, the proof does not extend to the case in which the basis vectors enclose arbitrary angles.
Is there a similar bound for nearly orthogonal bases constructed with the Lenstra–Lenstra–Lovász (LLL) lattice basis reduction algorithm?

\vspace{-0.0in}
\subsubsection*{Acknowledgements}
{\footnotesize
  The authors thank Georg Osang and Dmitriy Morozov for their motivating words, Alexey Garber, J\'{a}nos Pach, Morteza Saghafia, Nicolò Zava and Manuel Soriano Trigueros for technical discussions on the topic of this paper, and most of all Chiara Martyka for discussions on lattices and software implementing an earlier version of the algorithm.
}


\appendix

\clearpage
\section{Notation}
\label{app:N}

\begin{table}[h!]
  \centering \vspace{-0.1in}
  \begin{tabular}{ll}
     $\Rspace^d$
       &  $d$-dimensional Euclidean space \\
     $\sigma, \tau \in K$
       &  cells, cell complex \\
     $K_{t} = \filter^{-1} (-\infty, t]$
       &  sublevel set \\
     $\filter \colon K \to \Rspace$
       &  filter \\
     $\Height \colon \MTree{} \to \Rspace$
       &  height function on merge tree \\
     $\Frequency \colon \MTree{} \to \Monomial [R]$
       &  frequency function to the set of monomials \\
    \\
     $\Lambda = \Lambda(u_1,\ldots,u_d); U$
       &  lattice spanned by vectors; basis matrix \\
     $u, w \in \Lambda, \lambda_i \in \Zspace,$
       &  lattice points; integer coefficients \\
     $\Span{\Lambda}; \volume{p}{\Lambda}, \nu_p$
       &  linear space; volume of unit cell and unit ball \\
     $\filter/\Lambda \colon K/\Lambda \to \Rspace$
       &  quotient filter on quotient complex \\
     $\AAA, \BBB, \CCC$
       &  connected components \\
     $\aaa, \bbb, \ccc$
       &  corresponding shadows \\
     $\Lambda_\CCC, \Shadow{\CCC}{R}$
       &  periodicity lattice, shadow monomial \\
    \\
     $\Lambda_\CCC = \Lambda_\AAA + \Lambda_\BBB$
       &  sum, smallest common superlattice \\
     $r, s; x, y, z$
       &  roots; points, vertices, or nodes \\
     $\Root{x}, \Size{r}$
       &  root of component, \#vertices in component \\
     $\Next{x}$
       &  successor in linear list \\
     $a, \Shift{a}, \Drift{x}$
       &  arc, shift vector, drift vector of path from root \\
     $n, m$
       &  number of vertices, edges in quotient complex \\
     $V$; $v_1, \ldots, v_p$
       &  integer basis matrix; integer vectors \\
     $M[i,j]; \norm{M}_\infty; \multiplicity$
       &  integer basis matrix; max magnitude; multiplicity \\
    \\
     $\MTree = \perMTree{\filter}{\Lambda}$
       &  periodic merge tree \\
     $\MTree, \MTree', \MTree''$
       &  splintering periodic merge trees \\
     $\NTree, \NTree', \OTree$
       &  extra periodic merge trees \\
     $\varphi, \psi; \omega$
       &  maps between periodic merge trees \\
     $\MTree \simeq \MTree', [\MTree]$
       &  equivalence between, class of periodic merge trees\\
     $\Idist{\MTree}{\MTree'}$
       &  interleaving distance \\
     $\Jdist{[\MTree]}{[\MTree']}$
       &  interleaving pseudo-distance \\
     \\
     $\xi, \eta, \zeta \colon \Rspace \times \barRspace \to \Rspace$
       &  multiplicity functions \\
     $W_1(\xi, \eta), W_1^{\pm}(\xi, \eta)$
       &  (alternating) $1$-Wasserstein distance \\
     $\Tplan, \Xplan, \Yplan$
       &  transportation plans used in the optimization \\
     $\xi = \xi^+ - \xi^-$
       &  partition into positive, negative multiplicities \\
     $\BCode = \perBCode{\filter}{\Lambda}$
       &  periodic $0$-th barcode \\
     $\multiplicity_0, \multiplicity( \filter, \Lambda )$
       &  upper bound, maximum multiplicity of bars
  \end{tabular}
  \caption{Notation used in the paper.}
  \label{tbl:Notation}
\end{table}

\end{document}